\pdfoutput=1

\documentclass{article}
\usepackage{microtype}
\usepackage{graphicx}
\usepackage{subfigure}
\usepackage{booktabs}

\usepackage[accepted]{icml2024}

\usepackage{amsmath}
\usepackage{amssymb}
\usepackage{mathtools}
\usepackage{amsthm}

\usepackage{xurl}
\usepackage{hyperref}
\hypersetup{colorlinks=true, allcolors=blue, linktoc = all, linktocpage=true}
\usepackage[capitalize,nameinlink,noabbrev]{cleveref}
\crefname{equation}{}{}

\usepackage{multicol}
\usepackage{silence}
\usepackage{amsfonts,thmtools}
\usepackage{xparse}
\usepackage{xargs}
\usepackage{enumitem}
\usepackage{etoolbox}
\usepackage{mathtools}
\usepackage{complexity}
\usepackage{svg}
\usepackage{soul}

\usepackage{tablefootnote}
\newcommand{\tablefootnotemark}[1]{\textsuperscript{\getrefnumber{#1}}}

\usepackage{caption} 
\usepackage{booktabs}
\usepackage[bb=boondox]{mathalpha}
\definecolor{lightgray}{rgb}{0.9,0.9,0.9}
\definecolor{darkgray}{rgb}{0.5,0.5,0.5}
\sethlcolor{lightgray}

\newcommand\blfootnote[1]{%
  \begingroup
  \renewcommand\thefootnote{}\footnote{#1}%
  \addtocounter{footnote}{-1}%
  \endgroup
}
\usepackage{thmtools, thm-restate}
\declaretheorem{theorem}
\declaretheorem[sibling=theorem]{lemma}
\declaretheorem[sibling=theorem]{remark}

\declaretheorem[sibling=theorem]{proposition}
\declaretheorem[sibling=theorem]{corollary}
\declaretheorem[sibling=theorem]{definition}

\input{definitions.tex}
\usepackage{todonotes}
\newif\ifTODO 

\TODOfalse

\ifTODO
\paperwidth=\dimexpr \paperwidth + 6cm\relax
\oddsidemargin=\dimexpr\oddsidemargin + 3cm\relax
\evensidemargin=\dimexpr\evensidemargin + 3cm\relax
\marginparwidth=\dimexpr \marginparwidth + 3cm\relax
\fi

\usepackage{algorithm}
\usepackage{algcompatible}
\algnewcommand{\lst}{\texttt{lst}}
\algnewcommand{\slst}{\texttt{slst}}
\algnewcommand{\SEND}{\textbf{send}}

\newsavebox{\algleft}
\newsavebox{\algright}

\renewcommand{\algorithmiccomment}[1]{\hfill  $\diamond$ #1}

\renewcommand{\algorithmicrequire}{\textbf{Input:}}
\renewcommand{\algorithmicensure}{\textbf{Output:}}

\makeatletter
\newcounter{algorithmicH}
\let\oldalgorithmic\algorithmic
\renewcommand{\algorithmic}{%
  \stepcounter{algorithmicH}
  \oldalgorithmic}
\renewcommand{\theHALG@line}{ALG@line.\thealgorithmicH.\arabic{ALG@line}}
\makeatother



\icmltitlerunning{Convergence and Trade-Offs in Riemannian Gradient Descent and Riemannian Proximal Point}

\begin{document}

\twocolumn[

\icmltitle{Convergence and Trade-Offs in Riemannian Gradient Descent and Riemannian Proximal Point}

\icmlsetsymbol{equal}{*}

\begin{icmlauthorlist}
\icmlauthor{David Martínez-Rubio$^\ast$}{tub,zib}
\icmlauthor{Christophe Roux$^\ast$}{tub,zib}
\icmlauthor{Sebastian Pokutta}{tub,zib}
\end{icmlauthorlist}

\icmlaffiliation{zib}{Zuse Institute Berlin, Germany}
\icmlaffiliation{tub}{Technische Universität Berlin, Germany}

\icmlcorrespondingauthor{Christophe Roux}{roux@zib.de}
\icmlcorrespondingauthor{David Martínez-Rubio}{martinez-rubio.zib.de}

\icmlkeywords{Machine Learning}

\vskip 0.3in
]

\printAffiliationsAndNotice{$^\ast$Equal contribution}

\begin{abstract}
    In this work, we analyze two of the most fundamental algorithms in geodesically convex optimization: Riemannian gradient descent and (possibly inexact) Riemannian proximal point. We quantify their rates of convergence and produce different variants with several trade-offs. Crucially, we show the iterates naturally stay in a ball around an optimizer, of radius depending on the initial distance and, in some cases, on the curvature. In contrast, except for limited cases, previous works bounded the maximum distance between iterates and an optimizer only by assumption, leading to incomplete analyses and unquantified rates. 
    We also provide an implementable inexact proximal point algorithm yielding new results on minmax problems, and we prove several new useful properties of Riemannian proximal methods: they work when positive curvature is present, the proximal operator does not move points away from any optimizer, and we quantify the smoothness of its induced Moreau envelope. Further, we explore beyond our theory with empirical tests.

\end{abstract}

\section{Introduction}
\blfootnote{\color{darkgray}Most of the notations in this work have a link to their definitions, using \href{https://damaru2.github.io/general/notations_with_links/}{this code}, such as ${\protect\hyperlink{def:riemannian_exponential_map}{\color{darkgray}\operatorname{Exp}_{x}}}(\cdot)$, which links to where it is defined as the exponential map of a Riemannian manifold.}
Riemannian optimization is the study of optimizing functions defined over Riemannian manifolds.
This paradigm is used in cases that naturally present Riemannian constraints, which allows for exploiting the geometric structure of our problem, and for transforming it into an unconstrained one by working in the manifold. In addition, there are non-convex Euclidean problems, such as operator scaling \citep{allen2018operator} that, when phrased over a Riemannian manifold with the right metric, become convex when restricted to every geodesic, that is, they are geodesically convex (g-convex) \citep{neto2006convex, bento2012subgradient, bento2015proximal}.

Some other applications in machine learning are Gaussian mixture models \citep{hosseini2015matrix}, Karcher mean \citep{zhang2016fast}, dictionary learning \citep{cherian2016riemannian,sun2016complete}, low-rank matrix completion 
\citep{vandereycken2013low,mishra2014r3mc,tan2014riemannian,DBLP:journals/siamsc/CambierA16,heidel2018riemannian}, and optimization under orthogonality constraints \citep{edelman1998geometry,DBLP:conf/icml/CasadoM19}.
Riemannian optimization is a wide, active area of research, and numerous methods, such as the following first-order algorithms have been designed: projection-free \citep{weber2017frank,weber2019nonconvex}, accelerated \citep{martinez2020global, kim2022accelerated, martinez2023acceleratedmin}, min-max \citep{zhang2022minimax, jordan2022first, martinez2023acceleratedminmax, cai2023curvature}, stochastic \citep{tripuraneni2018averaging,khuzani2017stochastic,hosseini2019alternative}, and in particular variance-reduced methods \citep{zhang2016fast,sato2017riemannian,kasai2018riemannian}, among many others. 

A recurrent problem in Riemannian optimization algorithms is that geometric deformations appearing in their analyses scale with the distance between the iterates and between those and an optimizer. These distances are often bounded and quantified only by assumption \citep{zhang2016first,zhang2016fast,zhang2018towards,ahn2020nesterov,kim2022accelerated,zhang2022minimax,jordan2022first}. This assumption is the following: \textit{there is compact g-convex set $\mathcal{X}$, that the algorithm has a priori access to, in which the iterates stay, i.e., $x_t\in \mathcal{X}$}.

On the other hand, some works obtain convergence rates which are seemingly independent of the curvature, but they make use of conditions like smoothness or strong convexity without specifying where these have to hold, see \citep{smith1994optimization, udriste1994convex, cai2023curvature}  among others.
    This can be a problem, since unlike in the Euclidean space, where we can have globally smooth and strongly g-convex functions with constant condition number, in many Riemannian manifolds the condition number is lower bounded by a value that depends on the curvature and the diameter of the optimization domain \citep{martinez2020global, criscitiello2022negative}. For this reason, in order to quantify convergence rates, one has to assume problem parameters such as smoothness or strong g-convexity hold in a specific region where the iterates lie. This issue often leads to an unfinished argument due to a circularity in which the step sizes depend on the problem parameters, that depend on the set $\X$ where the iterates lie, that depends on the step sizes. See also \citet[Section 19.6]{hosseini2020recent} for a description of these issues and other references that suffer from it.

    One way of tackling these two problems is showing that our algorithms naturally stay in a bounded region that we can quantify.
    \citet{martinez2020global} presents an algorithm with this property, that reduces unconstrained g-convex problems to a sequence of problems in Riemannian balls of constant diameter.
    In the context of g-convex g-concave optimization, \cite{martinez2023acceleratedminmax} proved this property holds for an extragradient algorithm and \cite{wang2023riemannian, hu2023extragradient} showed it for other related algorithms. The latter work applies to the more general variational inequalities setting. An alternative approach is to add in-manifold constraints to the problem and design methods that can enforce those constraints. Projection-free algorithms like those of \cite{weber2017frank, weber2019nonconvex}, the Projected \RGD{} algorithms surveyed in \cref{sec:related-work}, and the accelerated constrained first-order methods in \cite{martinez2020global, martinez2023acceleratedmin, martinez2023acceleratedminmax} are designed to work with constraints and therefore they do not present the aforementioned problems.

Bounding the iterates of Riemannian algorithms has often been overlooked in the literature.
Because of this reason, the convergence of two of the most fundamental classes of first-order methods is not fully understood. One of them is Riemannian gradient descent (\newtarget{def:acronym_riemannian_gradient_descent}{\RGD{}}). Our first contribution is removing this limitation and quantifying the convergence rate and its dependence on geometric constants like the curvature. Secondly, we study the Riemannian proximal point algorithm (\newtarget{def:acronym_riemannian_proximal_point_algorithm}{\RPPA{}}). We provide an inexact version (\newtarget{def:acronym_riemannian_inexact_proximal_point_algorithm}{\RIPPA}) of it and convergence rates in general manifolds. Thirdly, for smooth functions we show how to implement the criterion of \RIPPA{} in different ways and provide some variants of \RGD{}. The result yields several algorithms with different convergence rates, where we trade off some dependence on the curvature for another or for optimizing in a larger set.
The latter entails, for instance, greater lower bounds on the condition number $\L/\mu$ for $\mu$-strongly g-convex $\L$-smooth functions in the set.
Lastly, we prove several new useful properties of Riemannian proximal methods.
More precisely, our contributions are summarized in the following and in \cref{table:summary}.

\begin{itemize}
    \item \textbf{\RGD{}}: Among other results, we show that for g-convex $\L$-smooth Riemannian functions with a minimizer $\xast$, \RGD{} with step size $\eta = 1/\L$ stays in a closed ball $\ball(\xast, \bigo{\zeta_{\R} \R})$, where $\R \defi \dist(\xinit, \xast)$ and $\zeta_{\R}$ is a geometric constant. If instead we use a step size $\eta = 1/(\zeta_{\bigo{\R}}\L)$, the iterates stay in $\ball(\xast, \bigo{\R})$. We quantify the rates of \RGD{} in different settings as a result. A composite \RGD{}, which implies reduced gradient complexity of solving the inexact prox in \RIPPA{}.

  \item \textbf{\RPPA{}}: A general analysis of \RPPA{}. It was only known in Hadamard manifolds before. An inexact \RPPA{} (\RIPPA{}) for both minimization and min-max problems and an implementation of it with first-order methods for smooth g-convex or g-convex-concave functions with quantified dependence on the curvature. For min-max problems in Hadamard manifolds, this approach provides the first algorithm that does not require prior knowledge of $\R$.

  \item \textbf{Prox properties}: The prox is quasi-nonexpansive, and the Moreau envelope $M(x) \defi \min_{y\in\X}\{f(y) + \frac{1}{2\eta}\dist(x, y)^2\}$ is $(\zeta_{\diam(\X)}/\eta)$-smooth in $\X$.
    \item \textbf{Experiments}: Numerical tests exploring beyond our theory. We observe that \RGD{} presents a monotonic decrease in distance to an optimizer and show that \RIPPA{} is competitive.
\end{itemize}

A future direction of research is studying whether one efficient algorithm can obtain the best of all worlds. We note that in Hadamard manifolds our method \hyperref[prop:nearly_constant_subroutine]{RIPPA-CRGD} already obtains this rate in terms of gradient complexity and iterate bounds, at the expense of solving subproblems that could be hard. The other methods can be implemented efficiently but have worse gradient complexity or iterate bounds. Finding an efficient method achieving the best of all worlds is a fundamental open problem. For the hyperbolic space, we do obtain this efficiently.

\paragraph{Outline}
We begin by introducing relevant definitions and notation in \cref{sec:prelim-notation}.
Then we provide a detailed review of prior works on \RGD{} and \RPPA{} in \cref{sec:related-work}.
We present our new results regarding \RGD{} and Riemannian proximal methods in \cref{sec:results}.
Then we present some empirical results in \cref{sec:experiments} and a conclusion in \cref{sec:conclusion}.

\section{Preliminaries and notation}\label{sec:prelim-notation}
The following definitions in Riemannian geometry cover the concepts used in this work, cf. \citep{petersen2006riemannian, bacak2014convex}.
A Riemannian manifold $(\M,\mathfrak{g})$ is a real $C^{\infty}$ manifold $\M$ equipped with a metric $\mathfrak{g}$, which is a smoothly varying inner product.
For $x \in \M$, denote by $\newtarget{def:tangent_space}{\Tansp{x}}\M$ the tangent space of $\M$ at $x$.
For vectors $v,w  \in \Tansp{x}\M$, we use $\innp{v,w}_x$ and $\norm{v}_x \defi \sqrt{\innp{v,v}_x}$ for the metric's inner product and norm, and omit $x$ when it is clear from context.
A geodesic of length $\ell$ is a curve $\gamma : [0,\ell] \to \M$ of unit speed that is locally distance minimizing.
A space is uniquely geodesic if every two points in that space are connected by one and only one geodesic.

The exponential map $\newtarget{def:riemannian_exponential_map}{\expon{x}} : \Tansp{x}\M\to \M$ takes a point $x\in\M$, and a vector $v\in \Tansp{x}\M$ and returns the point $y$ we obtain from following the geodesic from $x$ in the direction $v$ for length $\norm{v}$, if this is possible.
We denote its inverse by $\exponinv{x}(\cdot)$. It is well defined for uniquely geodesic manifolds, so we have $\expon{x}(v) = y$ and $\exponinv{x}(y) = v$. We denote the distance between two points by $\newtarget{def:distance}{\dist}(x, y)$.
The manifold $\M$ comes with a natural parallel transport of vectors between tangent spaces, that formally is defined from the Levi-Civita connection $\nabla$.
In that case, we use $\newtarget{def:parallel_transport}{\Gamma{x}{y}}(v) \in \Tansp{y}\M$ to denote the parallel transport of a vector $v$ in $\Tansp{x} \M$  to $\Tansp{y}\M$ along the unique geodesic that connects $x$ to $y$.

The sectional curvature of a manifold $\mathcal{M}$ at a point $x\in\mathcal{M}$ for a $2$-dimensional space $V\subset \Tansp{x}\M$ is the Gauss curvature of $\expon{x}(V)$ at $x$.
We denote by $\newtarget{def:M_manifold_with_lower_bounded_curv}{\Riemmin}$ the set of uniquely geodesic Riemannian manifolds of sectional curvature lower bounded by $\kmin$ and by $\newtarget{def:M_manifold_with_lower_bounded_curv}{\Riem}$ the set of uniquely geodesic Riemannian manifolds of sectional curvature that is lower and upper bounded in $[\newtarget{def:minimum_sectional_curvature}{\kmin}, \newtarget{def:maximum_sectional_curvature}{\kmax}]$.

A set $\X$ is said to be g-convex if every two points are connected by a geodesic that remains in $\X$. We note that if a manifold $\M \in \Riem$ has some positive sectional curvature, it may not be allowed to have arbitrarily large diameter. For example, since we work with uniquely geodesic manifolds, if $\kmin > 0$, it is necessary that the diameter of the manifold is $<\pi/\sqrt{\kmin}$. Thus, take into account that when we assume a ball of a certain radius is in $\M$, if $\kmax > 0$ the radius may be restricted. This is not a limitation for instance in Hadamard manifolds, which are complete simply-connected Riemannian manifold of non-positive sectional curvature, and in particular are diffeomorphic to $\mathbb{R}^n$ and uniquely geodesic.

Let $\X$ be a uniquely geodesic g-convex set. A differentiable function is $\newtarget{def:strong_g_convexity_of_F}{\mu}$-strongly g-convex (resp., $\newtarget{def:riemannian_smoothness_of_F}{\L}$-smooth) in $\X$, if we have $\circled{1}$ (resp. $\circled{2}$) for any two points $x, y \in \X$:
    \[
        \frac{\mu\dist(x,y)^2}{2}{\circled{1}[\leq]}f(y) -f(x) - \innp{\nabla f(x), \exponinv{x}(y)}{\circled{2}[\leq]}\frac{\L\dist(x,y)^2}{2}
      \]
    The function is said to be g-convex if $\mu=0$.
    If we parametrize a geodesic joining $x$ and $y$ as the constant speed curve $\gamma:[0, 1] \to \M$ such that $\gamma(0)=x$ and $\gamma(1)=y$, we have that convexity can be written as $f(\gamma(t)) \leq tf(x) + (1-t)f(y)$ and this also applies to non-differentiable functions.
    A function $f$ is $\newtarget{def:Lipschitz_constant}{\Lips}$-Lipschitz in $\X$ if $|f(x) - f(y)| \leq \Lips \dist(x, y)$ for all $x, y \in \X$.

    Given a g-convex set $\X\subseteq \M$ for $\M \in \Riemmin$, we denote by $\newtarget{def:proper_lsc_convex_class}{\proxc[\X]}$ the class of functions $f:\M\rightarrow \mathbb{R}\cup \{+\infty\}$ which are proper, lower semicontinuous and g-convex in $\X$. We denote by $\newtarget{def:smooth_convex_class}{\smc} \subset \proxc[\X]$ the subclass of functions which are also differentiable in an open subset $\N\subset\M$ containing $\X$, and are $\L$-smooth and g-convex in $\X$. We denote by $\newtarget{def:smooth_strongly_convex_class}{\smsc} \subset \smc$ the subset of those functions that are $\mu$-strongly g-convex in $\X$.
Note the dependence on $\X$ is important, since the possible condition numbers for functions in $\smsc$ depends on $\X$.
We denote by $\newtarget{def:non_smooth_convex_class}{\nsmc} \subset \proxc[\X]$ the subclass of functions which are also $\Lips$-Lipschitz and g-convex in $\X$.
  Let $\M,\N \in \Riem$ and let $\X\subset \M\times \N$ be a g-convex set. Then we denote by $\newtarget{def:convex_concave_class}{\proxsp}$ the class of functions $f:\M\times \N \rightarrow \mathbb{R}\cup \{+\infty,-\infty \}$ such that $f(\cdot,y)$ and $-f(x,\cdot)$ are, respectively, g-convex in $\X$, proper and lower semicontinuous for all $(x,y)\in \X$.

Given $r>0$, and a manifold $\M\in \Riem$, we define the geometric constants $\newtarget{def:zeta_geometric_constant}{\zeta_r}\defi r\sqrt{\abs{\kmin}}\coth(r\sqrt{\abs{\kmin}}) = \Theta(1 + r\sqrt{\abs{\kmin}})$ if $\kmin < 0$  and $\zeta_r\defi 1$ otherwise, and $\newtarget{def:delta_geometric_constant}{\delta_r} \defi r\sqrt{\kmax}\cot(r\sqrt{\kmax})\leq 1$ if $\kmax > 0$ and $\delta_r\defi 1$ otherwise. It is $\delta_r \leq 1 \leq \zeta_r$. For a g-convex set $\X\subseteq \M$ of diameter bounded by $D$ and containing $x\in\M$, the function $\Phi_x(y) \defi \frac{1}{2}\dist(x, y)^2$ is $\delta_D$-strongly g-convex and $\zeta_D$-smooth in $\X$, cf. \cref{fact:hessian_of_riemannian_squared_distance}.

We define the indicator $\newtarget{def:indicator_function}{\indicator{\X}}(x)$ as $0$ if $x\in\X$ and $+\infty$ if $x\not\in\X$.
A metric-projection operator $\newtarget{def:projection_operator}{\proj}:\M\to\X$ onto a closed g-convex set $\X$ is a map satisfying $\dist(\proj(x), y) \leq \dist(x, y)$ for all $y \in \X$.
$\newtarget{def:closed_ball}{\ball}(x, r)$ is a closed Riemannian ball of center $x$ and radius $r$.
The big-$O$ notation $\newtarget{def:big_o_tilde}{\bigotilde{\cdot}}$ omits $\log$ factors.
    In this paper $\newtarget{def:initial_point}{\xinit} \in \X$ always represent initial points of the algorithm we consider in that context.
    We assume that the functions we optimize contain at least one minimizer or saddle point denoted by $\newtarget{def:optimizer}{\xast}$, and we denote the initial distance to it by $\newtarget{def:initial_distance}{\R}\defi \dist(\xinit,\xast)$. When we work with minmax problems, or in a formulation which represents min or minmax, we use the variable name $z$ instead of $x$.
A point $x$ is an $\newtarget{def:accuracy_epsilon}{\epsilon}$-minimizer of $f$ if $f(x) - f(\xast) \leq \epsilon$ and a point $z= (x,y)$ is an $\epsilon$-saddle point of $f$ if
 $  f(x,\tilde{y})-f(\tilde{x},y)\le \varepsilon $ for all $\tilde{x},\tilde{y}\in \X$.
    For an algorithm with iterates $x_i$ which runs for $T$ iterations, we define $\newtarget{def:max_distance_to_x_ast}{\Rmax} \defi \max_{i \in T}\dist(x_i, \xast)$.
The algorithms we analyze in this paper are the following
\RGD{}, with update rule:
\begin{equation}\label{eq:RGD_update_rule}
    x_{t+1} \gets \expon{x_t}(-\eta \nabla f(x_t)).
\end{equation}

Uniform geodesic averaging of the iterates $\{x_1,\ldots,x_T\}$ is defined recursively as
  \begin{equation}\label{eq:g-avg-1}
    \bar{x}_{t+1}\gets \expon{\bar{x}_t} \left( \frac{1}{t+1}\exponinv{\bar{x}_t}(x_{t+1}) \right)
  \end{equation}
  for $t\in \{1,\ldots,T-1\}$ where $\bar{x}_1\gets x_{1}$.
The metric-projected \RGD{} (\newtarget{def:acronym_metric_projected_riemannian_gradient_descent}{\PRGD{}}) update rule is 
\[
x_{t+1} \gets \proj(\expon{x_t}(-\eta \nabla f(x_t))).
\] 
And given an $\eta>0$, the \RPPA{} update rule is:
\begin{equation}\label{eq:rppa_update_rule}
    z_{t+1} \gets \prox(z_t),
\end{equation}
where for a g-convex function $f$, $\newtarget{def:prox}{\prox}(\bar{x})\defi \argmin_{z\in \M} \{f(z)+ \frac{1}{2\eta}\dist(z,\bar{x})^2\}$, if it exists, which is always the case in our setting. Alternatively $\prox(\bar{z})\defi \argmin_{x\in \M} \argmax_{y\in \N}\{f(x,y)+ \frac{1}{2\eta}\dist(x,\bar{x})^2- \frac{1}{2\eta}\dist(y,\bar{y})^2\}$ with $\bar{z}\defi(\bar{x},\bar{y})$, if $f$ is a g-convex-concave function.

\section{Related work}
\label{sec:related-work}

We limit this section to non-asymptotic analyses of the algorithms we discuss, with a few exceptions. 

\subsection[Riemannian Gradient Descent]{Riemannian Gradient Descent}
Unless we specify otherwise, \RGD{} refers to \eqref{eq:RGD_update_rule} and for a g-convex compact $\X$ of diameter $D$, it assumes $x_{t} \in \X$ for $t=0, \dots, T$, while properties like $\L$-smoothness are assumed to hold in $\X$. Previous works either take this property as an assumption, rely on projections to enforce a bound or incur slow converge rates.
In contrast, we ensure this property holds for $\X$ being a ball around a minimizer without using projections to enforce it.

    For $\smsc$, \citet[Thm. 4.4]{gabay1982minimizing} showed an analysis of \RGD{} with per-iteration descent factor of $1- \frac{\mu^2}{\L^2}$, but only in the limit.
\citet{smith1994optimization} presents an analysis of \RGD{} with rates $\bigotilde{\frac{\L^2}{\mu^2}}$ and \citet[Theorem 4.2]{udriste1994convex} obtained the better rate $\bigotilde{\frac{\L}{\mu}}$.
\citet{zhang2016first} present several results on stochastic or deterministic \RGD{}, under a variety of assumptions on the function. The rates depend on $\zeta_{\Rmax}$ but $\Rmax = \max_{t\in[T]} \dist(x_t, \xast)$ is not quantified.
They analyze \PRGD{} for non-smooth optimization, where they can use $D \geq \Rmax$. They claim a \PRGD{} analysis for smooth functions but the proof was found to be flawed \citep{martinez2023acceleratedmin}.
\cite{bento2016iterationcomplexity} obtained a curvature-independent rate of \RGD{} for $\smc$ when the manifold is of non-negative sectional curvature. In this case, $\zeta_r = 1$ for every $r>0$ so this result is an instance of the one in \cite{zhang2016first}.
\cite{ferreira2019gradient} analyzed \RGD{} for $\smc$ but with some exponential constants depending on the sectional curvature and the initial gap.
\cite{martinez2023acceleratedmin} achieve linear rates of \PRGD{} for $\smsc$ assuming $\nabla f(\xast) = 0$ and $\zeta_D < 2$.
For $\smc$ they obtain the curvature-independent rates $\bigo{\frac{\L\Rmax[2]}{\epsilon}}$ for \RGD{}, but $\Rmax$ is not quantified. They also provide an analysis of \PRGD{} for $\smsc$ with a projection oracle that is not a metric projection, obtaining $\bigotilde{\frac{\L}{\mu}}$ rates.
\cite{martinez2023acceleratedminmax} present a general convergence analysis of \PRGD{} for $\smsc$ for Hadamard manifolds with rates depending on the Lipschitz constant of $f$ in $\X$, namely $\bigotilde{\frac{\L}{\mu}\zeta_C\zeta_D}$, for $D \defi \diam(\X)$ and $C \defi (\Lips/\L + 2D)/\zeta_D$. If $\nabla f(\xast)=0$, it is $\zeta_C = \bigo{1}$.

In this work we show convergence rates of different variants of unconstrained \RGD{}, providing different trade-offs. Showing that the iterates stay naturally bounded in a set whose diameter we quantify allowed us to bound $\Rmax$ in a principled way instead resorting to assuming such a bound.

\subsection{Riemannian Proximal Methods}
To the best of our knowledge, the first work on the Riemannian proximal point algorithm is due to  \cite{ferreira2002proximal}, with an asymptotic convergence in Hadamard manifolds with an exact proximal operator. They also established some properties of the algorithm in these manifolds.

There are numerous works on asymptotic convergence of exact or inexact \RPPA{} for g-convex optimization or more in general for variational inequalities, but we focus on discussing works with convergence rates.
\cite{bavcak2013proximal} obtained rates for \RPPA{} in Hadamard manifolds, and more generally for CAT(0) metric spaces, analogous to the classical Euclidean rates.  Under a growth condition, \cite{tang2014rate} present linear rates for an inexact \RPPA{} for a monotone operator $F$ in Hadamard manifolds. \cite{bento2016iterationcomplexity} rediscover the results of \cite{bavcak2013proximal} regarding the convergence rates for \RPPA{} in Hadamard.
\citet{bento2016new} obtains asymptotic convergence of \RPPA{} under Kurdyka–Lojasiewicz inequality, without assuming the manifold is Hadamard. \citet{espinola2016proximal, kimura2017proximal} also work in the general Riemannian case and obtain non-asymptotic convergence of an \RPPA{}, but with a convolving function that is not the distance squared.
In this work, we provide non-asymptotic rates for inexact \RPPA{} in the general Riemannian case, which are the first of their kind when allowing positive sectional curvature, and we show how this framework can be implemented with first-order methods in the g-convex smooth case. We present empirical results in \cref{sec:experiments}.

\subsection{Riemannian Min-Max Methods}
\label{sec:riemannian-min-max}
The methods discussed in this section apply to $\L$-smooth and g-convex-concave problems in the region where the iterates lie.
\citet{zhang2022minimax} introduced a variant of the Extragradient algorithm with rate $\bigo{\sqrt{\zeta_{\Rmax}/\delta_{\Rmax}}\L\Rmax[2]/\varepsilon}$ and \citet{jordan2022first} showed convergence rates of $\bigotilde{\sqrt{\zeta_{\Rmax}\delta_{\Rmax}}\L/\mu+\delta_{\Rmax}^{-1}}$ for the $\mu$-strongly g-convex-concave setting using the same algorithm. Both works did not bound $\Rmax$.
\citet{martinez2023acceleratedminmax} showed that with slightly modified step sizes the same algorithm satisfies $\Rmax=\bigo{\R}$ and the same convergence rates. Further, they introduced a proximal point algorithm for Hadamard manifolds, achieving accelerated fine-grained rates and being able to enforce a bound on $\Rmax$ via constraints.
\citet{wang2023riemannian} introduced an min-max algorithm based on optimistic online optimization and \citet{hu2023extragradient} present two variants of the Extragradient algorithm focusing on last-iterate convergence. Both of these works achieve an average-iterate rate of $\bigo{\zeta_{\R} \L\R^2/(\delta_{\R} \varepsilon})$ with $\Rmax=\bigo{\R}$.
Lastly \citet{cai2023curvature} shows a $\bigotilde{\frac{\L^2}{\mu^2}}$ convergence rate in terms of the gradient norm for the $\mu$-strongly g-convex-concave setting using \newtarget{def:acronym_riemannian_gradient_descent_ascent}{\RGDA{}}, a gradient descent-ascent algorithm. They do not bound $\Rmax$.

All of the discussed works except for the method presented in \citep{cai2023curvature} require prior knowledge of $\R$ in order to choose their step size, while the algorithm \citep{cai2023curvature} does not apply to non strongly g-convex-concave problems, and is suboptimal in $\L$ and $\mu$.
In this work, we show for Hadamard manifolds that implementing \RIPPA{} with \RGDA{} does not require prior knowledge of $\R$, while we can bound $\Rmax$. This comes at the cost of a worse dependency on $\zeta$ than algorithms that assume knowledge of $\R$.

\section{Convergence\,Results\,and\,Bounded\,Iterates}\label{sec:results}
We summarize the main convergence results presented in this section in \cref{table:summary}. We include the proofs in the appendix.
Consider as an example a general Hadamard manifold $\H$.
For a point $x\in \H$, for any $r > 0$, and for the ball $\ball(x, r)$, we have that $\Phi_x(y) \defi \frac{1}{2}\dist(x, y)^2$ is $\bigo{\zeta_r}$-smooth and $1$-strongly convex, cf. \cref{fact:hessian_of_riemannian_squared_distance}.
Using this fact and $\zeta_{\bigo{\R\zeta_\R}} =\bigo{\zeta_\R^2}$, we have that for $\H$, the expressions for the rates in \cref{table:summary} for strongly g-convex smooth functions of \RGD{} with both $\eta = \L^{-1}$ and $\eta = (\L\zeta_{\bigo{\R}})^{-1}$ are both $\bigotilde{\zeta_\R^2}$, despite of the seemingly better rate of the former. We note that for \RGD{} in the hyperbolic space, we obtained better convergence rates than for the general case, namely $\bigo{\frac{LR^2}{\epsilon}}$, $\bigotilde{\frac{\L}{\mu}}$, and $D = \bigo{R}$, respectively, matching the Euclidean rates.
\renewcommand*{\thefootnote}{\fnsymbol{footnotetable}}
\renewcommand*{\thefootnote}{\fnsymbol{footnote}}

\begin{table}[ht!]
    \caption{Summary of the convergence results in this work for g-convex functions in a ball $\X$ of diameter $D$ centered at $\xast$.
    All iterates stay in $\X$. Note that $\mu$, $\L$ and the Lipschitz constant $\Lips$ depend on the respective different sets $\X$, so $\L$ in two rows need not mean the same. The value $\eta >0$ is a proximal parameter.}
  \vspace{0.5em}
\begin{tabular}{@{}lcc@{}c@{}}
\toprule\addlinespace[0.3em]
  Method & $\mu=0$ & $\mu>0$ & $D$ \\ \midrule\addlinespace[0.3em]
  \multicolumn{4}{c}{$\L$\textsc{-smooth}}\\
 \midrule\addlinespace[0.7em]
     \hyperref[corol:conv_rate_RGD_1_L]{RGD$_{L^{-1}}$} & $\bigo{ \zeta_{\R}^2\frac{\L\R^2}{\epsilon} }$ & $\bigotilde{\frac{\L}{\mu}}$ &  $\bigo{\R\zeta_{\R}}$  \\\addlinespace[0.7em]
    \tablefootnote[1]{This is the rate for Hadamard manifolds only, for the general case see \cref{remark:reduce_to_str_cvx_for_RGD}.}\hyperref[remark:reduce_to_str_cvx_for_RGD]{Red. RGD$_{L^{-1}}$} & $\bigotilde{\zeta_{\R}^2{+}\frac{\L\R^2}{\epsilon} }$ & -- &  $\bigo{\R\zeta_{\R}}$  \\\addlinespace[0.7em]
     \hyperref[lem:rgd-zeta]{RGD$_{L^{-1}\oldzeta_{\bigo{R}}^{-1}}$} &  $\bigo{\zeta_{\R} \frac{\L\R^2}{\epsilon} }$&  $\bigotilde{ \zeta_{\R}\frac{ \L}{\mu}}$& $\bigo{\R}$   \\\addlinespace[0.7em]
     \hyperref[prop:nearly_constant_subroutine]{RIPPA-CRGD} & $\bigotilde{\frac{\L\R^2}{\delta_{2\R}\epsilon}}$ & $\bigotilde{\frac{\L}{\delta_{2\R}\mu}}$ &  $\bigo{\R}$  \\\addlinespace[0.7em]
      \tablefootnote[2]{\label{footnote:only_hadamard}These two results only apply to Hadamard manifolds.}\hyperref[prop:nearly_constant_subroutine]{RIPPA-PRGD}& $\bigo{\zeta_{\R}^2\frac{\L\R^2}{\epsilon}}$  & $\bigotilde{\zeta_{\R}^2\frac{\L}{\mu}}$ &  $\bigo{\R}$  \\\addlinespace[0.7em]
 \midrule\addlinespace[0.3em]
  \multicolumn{4}{c}{\textsc{non-smooth}}\\
 \midrule\addlinespace[0.7em]
    \hyperref[lem:rgd-nonsmooth]{RGD~NSm} &  $\bigo{\zeta_{\R} \frac{\Lips[2] \R^2}{\epsilon^2} }$& -- & $\bigo{\R}$   \\\addlinespace[0.7em]
\hyperref[thm:rippa]{RIPPA}&$\bigo{\frac{\R^2}{\eta\epsilon}}$&$\bigotilde{1+\frac{1}{\mu\eta}}$&$\bigo{\R}$\\\addlinespace[0.7em]
 \midrule\addlinespace[0.3em]
  \multicolumn{4}{c}{\textsc{min-max}}\\
 \midrule\addlinespace[0.7em]
\tablefootnotemark{footnote:only_hadamard}\hyperref[prop:rippa-minmax-rgd]{RIPPA-RGDA}&$\bigotilde{\zeta_{\R}^4\frac{\L\R^2}{\epsilon}}$&$\bigotilde{\zeta_{\R}^4\frac{\L}{\mu}}$&$\bigo{\R\zeta_{\R}}$\\\addlinespace[0.7em]
\bottomrule
\end{tabular}
    \label{table:summary}
\end{table}
\renewcommand*{\thefootnote}{\arabic{footnote}}

\subsection[Riemannian Gradient Descent]{Riemannian Gradient Descent} \label{sec:rgd}
We start by showing that for g-convex $\L$-smooth functions, the iterates of \RGD{} with the standard $\eta=1/\L$ step size naturally stay in a Riemannian ball around the optimizer. In the proof, we perform a careful analysis of the different terms playing a role in the convergence in order to bound the distances. We use $\varphi \defi (1+\sqrt{5})/2$.
\begin{restatable}{theorem}{iterateboundednessRGDL}\label{lemma:iterate_boundedness_RGD_1_L}\linktoproof{lemma:iterate_boundedness_RGD_1_L}
    Consider a manifold $\Mmin\in\Riemmin$, and $f\in\smc$ for $\X\defi\ball(\xast, \varphi\R\zeta_{\R}) \subset \Mmin$.  The iterates of \RGD{} with $\eta = 1/\L$ satisfy $x_t \in \X$. In addition, if $\Mmin$ is a hyperbolic space, and $\X_{\H} \defi \ball(\xast, \varphi\R)$, $f\in \smc[\X_{\H}]$, then $x_t\in\X_{\H}$.
\end{restatable}

This result allows us to fully quantify the convergence rate of \RGD{}, without resorting to assumptions about the distances of the iterates to the optimizers, as shown in the following.

\begin{restatable}{proposition}{convrateRGDL}\label{corol:conv_rate_RGD_1_L}\linktoproof{corol:conv_rate_RGD_1_L}
    Under the assumptions of \cref{lemma:iterate_boundedness_RGD_1_L}, we obtain an $\epsilon$-minimizer in $\bigo{\zeta_{\R}^2\frac{\L\R^2}{\epsilon}}$ iterations, or in $\bigo{\frac{\L\R^2}{\epsilon}}$ for the hyperbolic space. If $f$ is also $\mu$-strongly g-convex in $\X$, resp. in $\X_{\H}$, it takes $\bigo{\frac{\L}{\mu}\log(\frac{\L\R^2}{\epsilon})}$ iterations.
\end{restatable}

We also note that \RGD{} with any step size $< 2/\L$ never increases the function value, and so in fact, we only need to assume smoothness and g-convexity in the intersection of $\X$ and the level set of $f$ with respect to $x_0$. We discuss the size of the level set and convergence results which assume these properties hold in the level set in \cref{remark:level_sets}.

Interestingly, the rate we obtain in \cref{corol:conv_rate_RGD_1_L} by using our iterate bounds coincides with both the rate obtained from the curvature-dependent rate $\bigo{\zeta_{\Rmax}\frac{\L\R^2}{\epsilon}}$ in \cite{zhang2016first} and the seemingly curvature-independent rate $\bigo{\frac{\L\Rmax[2]}{\epsilon}}$ in \cite{martinez2023acceleratedmin} for general manifolds and for the hyperbolic space.
This fact highlights the importance of providing iterate bounds to fully quantify convergence rates.

We note that among all the algorithms in \cref{table:summary} applying to smooth functions, \RGD{} with $\eta=1/\L$ is the only one that does not require knowing the initial distance to a minimizer or a bound of it.
If we know $\R$ or an upper bound thereof, we can reduce the minimization of a function $f\in \smc$ to minimizing the strongly g-convex function $F(x) \defi f(x) + \frac{\epsilon}{2\R^2}\dist(\xinit, x)^2$. Indeed, applying \RGD{} with $\eta = 1/\L$ on $F(x)$, we obtain rates $\bigotilde{\zeta_{\R}^2 + \frac{\hat{L}\R^2}{\epsilon}}$ to find an $\epsilon$-minimizer of $f$ defined in a Hadamard manifold, where $\hat{L}$ is the smoothness constant of $f$ in $\ball(\xinit, \bigo{\R\zeta_{\R}})$. We can also quantify the rate in the general case, see \cref{remark:reduce_to_str_cvx_for_RGD}.

Without some iterate boundedness like the one in \cref{lemma:iterate_boundedness_RGD_1_L} we do not know what rates this reduction would yield, or what step size we should use, even though we have curvature independent rates for strongly g-convex smooth problems. This occurs because the smoothness and condition number of the regularized function depend on the sets where the iterates lie and they increase with the diameter of this set.

Alternatively, we can use \RGD{} with a step size $\eta=1/(\L\zeta_{\bigo{\R}})$.
As for the reduction described above, an upper bound on $\R$ can be used instead of its value.
Using this step size, we show that the iterates do not move away from the minimizer more than an amount of the same order as the initial distance.
Note that this step size is not in general smaller than the one in \cref{lemma:iterate_boundedness_RGD_1_L}, since the smoothness constants in both step sizes are taken with respect to sets of different sizes and are not necessarily identical.
In fact, for the problem we implement in \cref{sec:experiments}, the step size given by \cref{lem:rgd-zeta} is slightly larger.
\begin{restatable}{theorem}{rgdzeta}\label{lem:rgd-zeta}\linktoproof{lem:rgd-zeta}
    Consider a manifold $\Mmin\in\Riemmin$. Let $f\in\smc$, for $\X \defi\ball(\xast, \R\sqrt{3/2}) \subset \Mmin$. The iterates of \RGD{} with step size $\eta\defi 1/(\zeta_{\sqrt{3/2}\R} \L)$ satisfy $x_t \in \X$. The convergence rate is $\bigo{ \zeta_{\R} \L \R^2/\epsilon}$.
    If $f$ is also $\mu$-strongly g-convex in $\X$, then it takes $\bigo{(\zeta_{\R} \L/\mu)\log ( \L\R^2/\epsilon )}$ iterations.
\end{restatable}
We can also extend our techniques to show that for g-convex and Lipschitz functions, the iterates of Riemannian subgradient descent move away from an optimizer by at most a $\sqrt{2}$ factor farther than the initial distance.
\begin{restatable}[Non-smooth \RGD{}]{theorem}{nonsmoothrgd}\label{lem:rgd-nonsmooth}\linktoproof{lem:rgd-nonsmooth}
Consider a manifold $\Mmin\in \Riem$ and $f:\Mmin\to\R$ that is $\Lips$-Lipschitz in $\X \defi \ball(\xast,\sqrt{2} \R) \subset\Mmin$.
    The iterates of Riemannian subgradient descent with $\eta\defi \R/(\Lips \sqrt{\zeta_{\sqrt{2}\R} T})$ lie in $\X$ and the geodesic average of the iterates, cf. \cref{eq:g-avg-1}, is an $\epsilon$-minimizer of $f$ after $\bigo{ \zeta_{\R}\Lips[2] \R^2/\epsilon^{2}}$ iterations.
\end{restatable}

We now present an analysis of a composite \RGD{} (\newtarget{def:acronym_composite_riemannian_gradient_descent}{\CRGD{}}) algorithm, of independent interest. This algorithm exploits the ability to solve a structured problem for improved convergence. 
Note that while the update rule of \CRGD{} requires just one call to the gradient oracle, its implementation could be a hard computational problem. The interest of this result is that it can yield better information-theoretical upper bounds on the gradient oracle complexity than other approaches. For instance, for the implementation of proximal subroutines for the optimization of functions in $\smc$, see \cref{prop:nearly_constant_subroutine}.
\begin{restatable}[Composite \RGD{}]{proposition}{rgdcomposite}\label{lemma:composite_rgd}\linktoproof{lemma:composite_rgd}
    Let $\M \in \Riem$ and let $\X \subset\M$ be closed and g-convex. Given $f\in\smc$, and $g\in\proxc[\X]$, such that $F \defi f +g$ is $\mu$-strongly g-convex in $\X$, and $x^\ast\defi\argmin_{x\in\X}F(x)$. Iterating the rule
\begin{equation*}
    x_{t+1}{\gets}\argmin_{y\in\X}\innp{\nabla f(x_t), \exponinv{x_t}(y)}{+}\frac{\L}{2}\dist(x_t, y)^2 + g(y),
\end{equation*}
    we get an $\epsilon$-minimizer of $F$ in $\bigo{\frac{\L}{\mu}\log(\frac{F(\xinit)-F(\xast)}{\epsilon})}$ iterations.
\end{restatable}
We note that in the proof of \cref{lemma:composite_rgd} we showed that the method above is well defined, in particular, that the argmin in the problem defining $x_{t+1}$ above exists.

\subsection{Riemannian Proximal Methods} \label{sec:rpm}
We present our results on proximal methods in a unified way for minimization and min-max problems. For Hadamard manifolds, it is known that the prox is a non-expansive operator, cf. \cref{sec:alternative_proofs}. However, that is not the case in general manifolds \citep[Section 6.1]{wang2023online}. Still, we are able to show that the iterates of \RPPA{} never move farther from an optimizer than the initial distance in general manifolds, meaning that the prox is a quasi-nonexpansive operator, which allows us to provide fully quantified rates of convergence of it and its inexact version.

\begin{restatable}[RPPA]{proposition}{exactRPPAvi}\label{prop:exact-rppa}\linktoproof{prop:exact-rppa}
  Consider either a manifold $\M \in \Riem$ and a function $f\in \proxc[\X]$ with $\X \defi \ball(\zast, 2\R) \subset\M$ or the manifolds $\M,\N\in \Riem$ and a function $f\in\proxsp[\X]$ with $\X \defi \ball(\zast,2\R)\subset \M\times\N$. For any $\eta > 0$ and all $t \geq 0$, the iterates of the exact \RPPA{}, cf. \eqref{eq:rppa_update_rule}, satisfy $\dist(z_{t+1}, \zast) \leq \dist(z_{t}, \zast)$. In particular, it is $z_t \in \X$.
\end{restatable}

Further, we show that if the iterates are computed inexactly as described in \cref{alg:rippa}, they only move away from an optimizer by a small constant factor from the initial distance, and we quantify the convergence rates.
We note that we can make the iterates stay in $\ball(\xast, r)$ for $r > \R$ as close as we want to $\R$, by making the criterion in Line \ref{line:ppa-criterion} more strict.

The convergence rates of \RPPA{} can be derived from the one of \RIPPA{} when setting the error to $0$, and in general they are the same up to constant factors. This convergence result is surprising, since for minimization, \RPPA{} is equivalent to \RGD{} on the Moreau envelope, cf. \cref{corol:grad_of_moreau_envelope}, which can be non g-convex when positive curvature is present, cf. \cref{rem:non_g_convex_moreau_envelope}. 

\begin{restatable}[\RIPPA{}]{theorem}{inexactRPPAvi}\linktoproof{thm:rippa}\label{thm:rippa}
    Consider either a manifold $\M \in \Riem$ and a function $f\in \proxc[\X]$ where $\X\defi\ball(\zast,5\R) \subset \M$ or the manifolds $\M,\N \in \Riem$ and a function $f\in\proxsp[\X]$ where $\X\defi\ball(\zast,5\R) \subset \M\times \N$. Using the notation in \cref{alg:rippa}, it holds that $z_t \in \ball(\zast,\sqrt{2}\R)$ for every $t \geq 0$, and the output of \cref{alg:rippa} after $T=\bigo{\frac{\R^2}{\eta \epsilon}}$ iterations is an $\epsilon$-minimizer of $f$ or an $\varepsilon$-saddle point of $f$.
    If $f$ is $\mu$-strongly convex or $\mu$-strongly convex-concave in $\ball(\zast,\sqrt{2}\R)$, then $\dist(z_{t+1}, \zast)^2 \leq \frac{1}{1+\eta\mu/2}\dist(z_t, \zast)^2$ and in particular $\dist(z_T, \zast)^2 \leq \varepsilon_d$ after $T=\bigo{(1+\frac{1}{\mu\eta})\log (\frac{\R^2}{\varepsilon_d} )}$ iterations.
\end{restatable}

{
\begin{algorithm}[ht!]
    \caption{Riemannian Inexact Proximal Point Algorithm ({\sc rippa})}
    \label{alg:rippa}
\begin{algorithmic}[1]
    \REQUIRE Manifolds $\M,\N \in \Riem$ and $\X\subset \M$ or $\X \subset \M\times\N$, initial point $\zinit \in \X$, $\mu$-strongly g-convex or g-concave function $f\in \proxc[\X]$ or $f\in \proxsp[\X]$, for $\mu \geq 0$, and proximal parameter $\eta > 0$.
    \vspace{0.1cm}
    \hrule
    \vspace{0.2cm}
    \hspace{-1.3cm}\textbf{Definitions:}
    \begin{itemize}[leftmargin=*]
        \item Exact prox: $z^\ast_{t+1}\defi \prox(z_t)$.
      \item For $f\in \proxc[\X]$: $F(z)\defi\partial f(z)$
            \item For $f\in\proxsp[\X]$: $F(z)\defi(\partial_x f(x,y),-\partial_y f(x,y))$
            \item Subgradient: $v_{t+1}\in F(z_{t+1})$.
        \item Error $r_{t+1}\defi\eta v_{t+1}-\exponinv{z_{t+1}}(z_t)$.
        \item For $\mu=0$: $\Delta_t \defi (t+1)^{-2}$. For $\mu>0$: $\Delta_t \defi \eta\mu/2$.
    \end{itemize}

    \vspace{0.4cm}
    \hrule
    \vspace{0.4cm}
    \FOR {$t = 0 \text{ \textbf{to} } T-1$}
    \vspace{1em}
    \State \label{line:ppa-criterion}
    \vspace{-3em}
    \begin{align*}
        z_{t+1}\gets& \text{approx. } \prox(z_t)\\
    \text{s.t. }\quad & \dist(z_{t+1},z^\ast_{t+1})^2\le \frac{1}{4} \dist(z_t,z_{t+1}^{*})^2,\\
     &\norm{r_{t+1}}^2\le \Delta_{t+1}\delta_{5\R} \dist(z_{t+1},z_t)^2.
    \end{align*}
    \vspace{-2em}
    \ENDFOR
    \ENSURE $z_T$ {\bf if} $\mu>0$, {\bf else} uniform geodesic averaging of $z_1, \dots z_{T}$, cf. \cref{cor:g-avg-1}
\end{algorithmic}
\end{algorithm}
}

While \cref{thm:rippa} does not require smoothness, we can exploit this condition to give an efficient implementation via first-order methods.
\begin{restatable}{proposition}{implementationRIPPA}\linktoproof{prop:nearly_constant_subroutine}\label{prop:nearly_constant_subroutine}
    In the setting of \cref{thm:rippa} for $f\in\smc[\X]$, that is, for $f$ g-convex and $\L$-smooth in $\X$, let $\eta = 1/\L$, $\X\defi\ball(\xast, 4\R)\subset\M$, . The composite Riemannian Gradient Descent of \cref{lemma:composite_rgd} in $\Xt \defi \ball(x_t, 2\R)$ implements the criterion in Line~\ref{line:ppa-criterion} of \cref{alg:rippa} at iteration $t$ using $\bigotilde{1/\delta_{4\R}}$ gradient oracle queries. If $\M$ is Hadamard, \PRGD{} in $\Xt$ implements the criterion after $\bigotilde{\zeta_{\R}^2}$ iterations.
\end{restatable}

Recall that for a Hadamard manifold, it is $\delta_{r} =1$, for all $r \geq 0$, so in this case \CRGD{} uses $\bigotilde{1}$ gradient queries, while implementing each step of \CRGD{} could be expensive. On the other hand \PRGD{} has a worse gradient complexity but can be implemented easily, since projection onto a ball can be done in closed form \cite{martinez2023acceleratedmin}. 
We note that the first iteration of \CRGD{} with the squared-distance regularization in \cref{prop:nearly_constant_subroutine} is equivalent to a step of \PRGD{} with a different step-size and thus it can be easily implemented, see \cref{sec:rippa_implementation}.

The following proposition shows that for smooth min-max problems on Hadamard manifolds, we can implement the criteria of \RIPPA{} using \RGDA{} \citep{cai2023curvature} without knowledge of $\R$ unlike prior works. While this implementation obtains the same optimal rates with respect to $\epsilon$, $\L$, and $\mu$ if it applies, its complexity in terms of $\zeta$ is worse than previous methods, cf. \cref{sec:riemannian-min-max}. This is a trade-off between requiring the knowledge of $\R$ and achieving better dependence on $\zeta$.

\begin{restatable}[Implementing Min-Max RIPPA]{proposition}{RippaRgdImplement}\label{prop:rippa-minmax-rgd}\linktoproof{prop:rippa-minmax-rgd}
    In the setting of \cref{thm:rippa} for $f\in\proxsp[\X]$, suppose that in addition $\M$, $\N$ are Hadamard Manifolds, $\X=\ball((\xast,\yast), 7R\zeta_{\R})\subseteq\M\times \N$, $f$ is $\L$-smooth in $\X$, and $\eta = 1/\L$.
Then \RGDA{} with step size $1/\L$ implements the criterion in \cref{line:ppa-criterion} at iteration $t$ using $\bigotilde{\zeta_{\R}^4}$ gradient oracle queries.
\end{restatable}
We conclude this section studying the smoothness of the Moreau envelope, tightly related to the proximal point algorithm. In particular, this algorithm is equivalent to performing \RGD{} on this envelope. This is a very useful tool in the design of optimization algorithms \citep{parikh2014proximal, davis2019stochastic}. First, we provide an expression for the gradient of the Moreau envelope.

\begin{restatable}[Gradient of Moreau envelope]{lemma}{gradientmoreauenvelope}\linktoproof{corol:grad_of_moreau_envelope}\label{corol:grad_of_moreau_envelope}
  Let $\M$ be a uniquely geodesically Riemannian manifold, let $\XX \subset \M$ be a g-convex closed set.
  For $f\in \proxc[\X]$, and the Moreau envelope of $g\defi f+\indicator{\X}$ with $\eta>0$, defined as 
\[
M(x) \defi \min_{z\in\M}\{f(z) + \indicator{\XX}(z) + \frac{1}{2\eta} \dist(x, z)^2\}.
\] 
    We have $\nabla M(x) = -\frac{1}{\eta}\exponinv{x}(\prox[g](x))$.
\end{restatable}

Now, we can provide a bound for the value of the Moreau envelope smoothness.

\begin{restatable}[Moreau envelope smoothness]{theorem}{smoothnessofmoreauenvelope}\linktoproof{prop:smoothness_of_moreau_envelope}\label{prop:smoothness_of_moreau_envelope}
    Consider $\Mmin\in\Riemmin$, and let $\mathcal{X} \subset \Mmin$ be a g-convex closed set. For $f \in \proxc$, we have that the Moreau envelope of $g\defi f+\indicator{\X}$ with parameter $\eta > 0$, defined for all $x\in\Mmin$ as $M (x) \defi \min_{z\in\Mmin}\{f(z) + \indicator{\X}(z) + \frac{1}{2\eta} \dist(x, z)^2\}$, satisfies for all $x, y\in\M$:
  \begin{equation*}
      \begin{aligned}
          M(y) &\leq M(x) + \innp{\nabla M(x), \exponinv{x}(y)}\\ 
          & \ \ + \frac{\zeta_{\dist(x, \prox[g](x))}}{2\eta}\dist(x, y)^2.
      \end{aligned}
  \end{equation*}
    In particular, if $\X$ is compact and its diameter is $D$, the Moreau envelope $M(x)$ is $(\zeta_D/\eta)$-smooth in $\X$.
\end{restatable}

We note that if $\kmin \geq 0$, the Moreau envelope is $(1/\eta)$-smooth. That is, in this case the smoothness is not degraded by the curvature with respect to the Euclidean case, while the g-convexity can be lost.

The intuition about the proof of \cref{prop:smoothness_of_moreau_envelope} is the following. The epigraph of the Moreau envelope can be seen as the union of the epigraphs $\{(x, f(y) + \indicator{\X}(y) + \frac{1}{2\eta}\dist(x,y)^2) \ | \ x \in \Mmin\}$ for all $y \in \Mmin$. Consequently, given the $\prox(x)$, we have that the quadratic $U(y) \defi f(x) + \frac{1}{2\eta}\dist(x, y)^2$ satisfies $U(y) \geq M(y)$, for all $y \in \X$. And in fact, in light of the definition of $M(\cdot)$ and of \cref{corol:grad_of_moreau_envelope} that shows $\nabla M(x) = -\frac{1}{\eta}\exponinv{x}(\prox[g](x))$, we have $U(x) = M(x)$ and $\nabla U(x) = \nabla M(x)$. The quadratic $U(\cdot)$ is itself smooth by the cosine inequality, cf. \cref{remark:tighter_cosine_inequality}, so it has a quadratic in $\Tansp{x}\Mmin$ whose induced function in $\Mmin$ upper bounds $M(\cdot)$. In the supplementary material we present a proof based on this intuition and then we present another analysis, that although suboptimal, shows a different point of view on the problem.

\begin{figure}[h]
\centering
    \hspace{-0.5cm}
    \includegraphics[width=1.05\columnwidth]{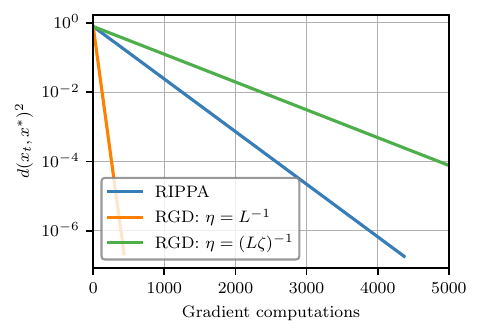}
    \caption{Comparison of \RIPPA{} and of \RGD{} with  $\eta=\L^{-1}$ for solving \cref{eq:karcher} in the hyperbolic space $\mathbb{H}^d$ with $n=1000$ centers and dimension $d=1000$ in terms of squared distance to the optimizer $\xast$. Smoothness $\L$ is taken for a set of diameter $\bigo{\R}$. We observe monotonous decrease in distance in all of our experiments.}
  \label{fig:loss_H}
\end{figure}

\section{Experiments}
\label{sec:experiments}

We present experimental results for computing the Karcher mean in the $d$-dimensional hyperbolic space $\mathbb{H}^d$, and the $d(d+1)/2$-dimensional manifold of symmetric positive definite matrices $\mathcal{S}_+^d\defi\{M\in \mathbb{R}^{d\times d}: M = M^T, M\succ 0\}$ with the affine invariant metric \citep{hosseini2015matrix}, which both have non-positive sectional curvature, see \cref{sec:experiment-details}.
We implement \RGD{} with step sizes $\eta =1/\L$ and $\eta=1/(\L\zeta_{\bigo{\R}})$ as well as \RIPPA{} performing a constant number of iterations of \PRGD{} to approximately solve the proximal problems. The first step can be taken as \CRGD{}, as we explained in \cref{prop:nearly_constant_subroutine}.

We implement this problem using the Pymanopt library \citep{townsend2016pymanopt}, published under the BSD-3-Clause license. We run until a fixed precision is reached in function value, and because of this, different algorithms stop at points at different distances from $\xast$. We provide the results in function value and more experiments in \cref{sec:numerical-results}, where we observe similar behavior.
The experiments show that (A) \RIPPA{} is a competitive algorithm for solving g-convex and smooth optimization problems and that (B) the distance of the iterates of \RGD{} and \RIPPA{} to the optimizer $\xast$ monotonically decrease in practice, which goes beyond what our theoretical results predict.

For $\M \in \Riem$, the Karcher mean is defined as
\begin{equation}
  \label{eq:karcher}
 \min_{x\in \M} \left\{F(x)\defi\frac{1}{2n}\sum_{i=1}^n\dist^2(x,y_i)\right\}.
\end{equation}
For a g-convex set $\X\subset\M$ containing all points $y_i$, the function $F$ is $\zeta_D$-smooth and $\delta_D$-strongly g-convex in $\X$, where $D\defi\operatorname{diam}(\X)$.
In this problem, we can analytically compute an upper bound on $\R$, which allows us to choose the step sizes in a principled manner, see \cref{sec:experiment-details}.

\cref{lemma:iterate_boundedness_RGD_1_L,lem:rgd-zeta,prop:nearly_constant_subroutine} guarantee that iterates naturally stay in a ball around the optimizers with a radius that is larger than the initial distance $\R\defi \dist(\xast,\xinit)$.
Based on these results, one might expect to see some increase in distance to the optimizer at some point.
However, in \cref{fig:loss_H,fig:loss_SPD} and in the results for different parameters in \cref{sec:numerical-results}, the distance of the iterates to the optimizer is monotonically decreasing for all algorithms.
In fact, we ran the algorithms for different settings, different initializations, and we performed a grid search on the step sizes.
In all instances except those in which the step size was too large and the algorithm diverged, the distances where monotonically decreasing.
This indicates that our bounds on the distance of the iterates to optimizers could potentially be improved.
Alternatively, it might be that the distance between the optimizers and the iterates only increases for some pathological functions which do not arise in practice.

In \cref{fig:loss_SPD}, \RIPPA{} outperforms both \RGD{} variants. The results for \RGD{} with both step sizes are similar, which may be due to the fact that for the Karcher mean the step sizes and the convergence rates coincide up to constant factors on any $\M\in \Riem$, see \cref{sec:experiment-details}.
This exemplifies the issue with stating smoothness and strong g-convexity constants of a function without specifying the size of the set in which they hold.
Due to \cref{lemma:iterate_boundedness_RGD_1_L}, we have that the iterates of \RGD{} with  $\eta=1/\L$ stay in $\ball(\xast,\varphi \R)$ in $\mathbb{H}^d$, which allows for larger step sizes than in general manifolds. This means that one would not expect $\eta=1/(\L\zeta_{\bigo{\R}})$ to provide an advantage in this setting, which is what we observe.
We do not have such a refined result for \RIPPA{} in hyperbolic space.
This may be why in this special case, we see in \cref{fig:loss_H} that \RGD{} with $\eta=1/\L$ converges significantly faster than \RIPPA{} while \RGD{} with $\eta=1/(\L\zeta_{\bigo{\R}})$ is the slowest method.

\begin{figure}[h]
\centering
    \hspace{-0.5cm}
    \includegraphics[width=1.05\columnwidth]{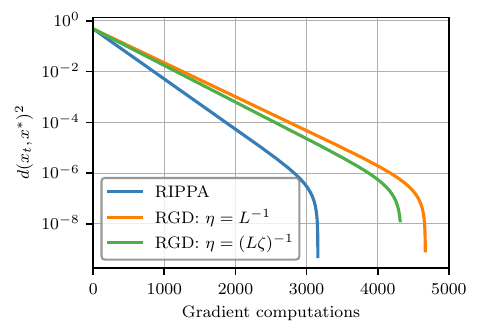}
    \caption{Comparison of \RIPPA{} and of \RGD{} with  $\eta=\L^{-1}$ and $\eta=(\L\zeta_{O(\R)})^{-1}$ for solving \cref{eq:karcher} in $\mathcal{S}_+^{100}$ with $n=1000$ centers, and dimension $d(d+1)/2=5050$ in terms of squared distance to the optimizer $\xast$. Smoothness $\L$ is taken for a set of diameter $\bigo{\R\zeta_{\R}}$. We observe monotonous decrease in distance in all of our experiments.}
  \label{fig:loss_SPD}
\end{figure}

\section{Conclusion and Discussion}
\label{sec:conclusion}

In spite of recent advances in Riemannian optimization, the interplay between the curvature of the manifolds and the behaviour of optimization algorithms is still not fully understood.
In this article, we advance the understanding on this connection for \RGD{} and \RPPA{} algorithms by providing full convergence rates without artificial assumptions, and different variants that enjoy different convergence rates, trading off its dependence on geometric constants with some guarantee on iterate boundedness or on efficiency of implementation.
Further, we presented the first analysis of the inexact Riemannian proximal point algorithm which holds in general manifolds. We provide non-asymptotic convergence guarantees and explicitly show that its iterates also stay in a set around an optimizer, and provided new properties of the Riemannian proximal operator.

One presented algorithm guarantees that the iterates move away from an optimizer at most a small constant factor farther than the initial iterate. An interesting future direction of research is studying whether this constant factor is unavoidable or if there is an \RGD{} algorithm which never moves away from an optimizer, as it is the case in the Euclidean space. Or more generally, whether the \RGD{} rule is a non-expansive operator for smooth g-convex functions for some choice of the step size. This would have implications to algorithmic stability and differential privacy.

For smooth functions, it is of interest to explore whether one can implement a Riemannian inexact proximal point algorithm by using a constant number of iterations in the subroutine, and therefore avoiding an extra logarithmic factor in the convergence results.
Additionally, a very interesting open question is the study of whether we can obtain an efficiently implementable algorithms that obtain the best of all of our rates and iterate bounds, and also possibly without assuming knowledge of the initial distance to an optimizer.

\section*{Acknowledgements}
This research was partially funded by the Research Campus Modal funded by the German Federal Ministry of Education and Research (fund numbers 05M14ZAM,05M20ZBM) as well as the Deutsche Forschungsgemeinschaft (DFG) through the DFG Cluster of Excellence MATH$^+$ (EXC-2046/1, project ID 390685689).
\bibliography{refs}
\bibliographystyle{icml2024}

\newpage
\appendix
\onecolumn

\section[RGD proofs]{\RGD{} proofs}

\iterateboundednessRGDL*

\begin{proof}\linkofproof{lemma:iterate_boundedness_RGD_1_L}
    Define $\varphi \defi (1+\sqrt{5})/2$ and $\zeta \defi \zeta_{\varphi \R\zeta_{\R}}$. We first check that $\zeta = \bigo{\zeta_{\R}^2}$. Since it is $\zeta_r \in [r\sqrt{\abs{\kmin}}, r\sqrt{\abs{\kmin}}+1]$ and $\zeta_r \geq 1$, for all $r \geq 0$, we have $\zeta_{\varphi \R\zeta_{\R}} \leq \varphi \R\sqrt{\abs{\kmin}}\zeta_{\R} + 1 \leq \varphi\zeta_{\R}^2 +1 = \bigo{\zeta_{\R}^2}$.
    Now denote $\Delta_i\defi f(x_i)-f(x^*)$. We show by induction that $\dist(x_{t}, \xast) \leq \varphi \R\zeta_{\R}$, for all $t \geq 0$. This holds for $t=0$ by definition so suppose the property holds for all $i \leq t$ and let's prove it for $t+1$. Let $\eta_\ast \defi \max_{\eta > 0} \{ \eta \ | \ \expon{x_t}(-\eta_\ast \nabla f(x_t)) \in \ball(\xast, \varphi \R\zeta_{\R})\}$. Note we can write a maximum because the ball is compact. It is enough to show $\eta_\ast \geq \frac{1}{\L}$. Suppose $\eta_\ast < \frac{1}{\L}$ and let $x_{t+1}' \defi \exponinv{x_t}(-\eta_\ast \nabla f(x_t))$. By definition, it must be $\dist(x_{t+1}', \xast) = \varphi \R \zeta_{\R}$. We will arrive to a contradiction.
We have for all $i < t$:
\begin{equation}
  \label{eq:rgd-descent}
    \Delta_{i+1}-\Delta_i \leq\innp{\nabla f(x_i), \exponinv{x_i}(x_{i+1})}+\frac{\L}{2} \dist\left(x_{i+1}, x_i\right)^2= -\frac{1}{2 \L}\norm{\nabla f(x_i)}^2,
\end{equation}
    where we used $\L$-smoothness of $f$ in $\ball(\xast, \varphi \R\zeta_{\R})$ the definition of $x_{i+1}$ for $i < t$, and the induction hypothesis that allows us to use the $\L$-smoothness property. Similarly, by the definition of $x_{t+1}'$ and defining $\Delta_{t+1}' \defi f(x_{t+1}')-f(\xast)$, we have:
\begin{equation}
  \label{eq:rgd-descent_last_it}
    \Delta_{t+1}'-\Delta_t \leq\innp{\nabla f(x_t), \exponinv{x_t}(x_{t+1}')}+\frac{\L}{2} \dist\left(x_{t+1}', x_t\right)^2=  \left(-\eta_\ast + \frac{\L\eta_\ast^2}{2}\right) \norm{\nabla f(x_t)}^2 \circled{1}[<] -\frac{\eta_\ast}{2}\norm{\nabla f(x_t)}^2,
\end{equation}
    where $\circled{1}$ is equivalent to $\eta_\ast \in (0, 1/\L)$. We also have the following bound, for all $i < t$:
\begin{equation}
  \label{eq:rgd-cvx}
\begin{aligned}
    \Delta_i \circled{1}[\leq]\innp{-\nabla f(x_i), \exponinv{x_i}(x^*)} &\circled{2}[\leq] \frac{\L}{2}\left[\dist(x_i, x^*)^2-\dist(x_{i+1}, x^*)^2 + \zeta \dist(x_i,x_{i+1})^2\right]\\
    &=\frac{\L}{2}\left[\dist(x_i, x^*)^2-\dist(x_{i+1}, x^*)^2 \right] + \frac{\zeta}{2\L}\norm{\nabla f(x_i)}^2,\\
\end{aligned}
\end{equation}
    where $\circled{1}$ uses g-convexity of $f$, $\circled{2}$ uses the cosine inequality \cref{remark:tighter_cosine_inequality} and the bound $\zeta_{\dist(x_i, \xast)} \leq \zeta$ which holds by induction hypothesis and monotonicity of $r \mapsto \zeta_r$. Likewise, we have
\begin{equation}
    \label{eq:rgd-cvx_last_it}
\begin{aligned}
    \Delta_{t} &\le\frac{1}{2\eta_\ast}\left[\dist(x_t, x^*)^2-\dist(x_{t+1}', x^*)^2 \right] + \frac{\zeta\eta_\ast}{2}\norm{\nabla f(x_t)}^2\\
\end{aligned}
\end{equation}
Multiplying \cref{eq:rgd-descent} by $\zeta$ and adding it to \cref{eq:rgd-cvx}, and similarly with \cref{eq:rgd-descent_last_it} and \cref{eq:rgd-cvx_last_it} we obtain:
\begin{align}
\begin{aligned} \label{eq:rgd-step}
    \zeta \Delta_{i+1} - (\zeta-1) \Delta_i &\leq \frac{\L}{2}\left(\dist(x_i, x^*)^2-\dist(x_{i+1}, x^*)^2\right) \text{ for } i < t \\
    \eta_\ast \L (\zeta \Delta_{t+1}' - (\zeta-1) \Delta_t) &< \frac{\L}{2}\left(\dist(x_t, x^*)^2-\dist(x_{t+1}', x^*)^2\right)
\end{aligned}
\end{align}
Adding up from $i=0$ to $t$, we obtain
    \begin{align}\label{eq:concluding_rgd_iterate_boundedness}
    \begin{aligned}
        \eta_\ast \L \zeta\Delta_{t+1}' + \zeta\Delta_t(1-\eta_\ast \L) + \eta_\ast \L \Delta_t + \sum_{i=1}^{t-1} \Delta_i  + \frac{\L}{2} \dist(x_{t+1}', \xast)^2 & < (\zeta-1) \Delta_0 + \frac{ \L \dist(\xinit, \xast)^2}{2} \\
        &\circled{1}[\leq] \frac{\zeta \L \R^2}{2}.
    \end{aligned}
\end{align}
    where $\circled{1}$ uses that by smoothness $\Delta_0 \leq \frac{\L \dist(\xinit, \xast)^2}{2} \leq \frac{\L \R^2}{2}$. Using $\eta_\ast \L \in (0, 1)$ dropping all the terms with $\Delta_i, \Delta_{t+1}' \geq 0$, and simplifying, we obtain $\circled{2}$ below, while the rest of the following holds by the definition of $\zeta$, and the fact that for all $r \geq 0$, we have $\zeta_r \geq 1$ and $\zeta_r\in [r\sqrt{\abs{\kmin}}, r\sqrt{\abs{\kmin}}+1]$:
\[
    \dist(x_{t+1}', \xast)^2 \circled{2}[<] \zeta \R^2 \leq (\varphi \zeta_{\R} \R \sqrt{\abs{\kmin}} +1) \R^2  \leq (\varphi \zeta_{\R}^2 + 1) \R^2 \leq  (\varphi+1) \zeta_{\R}^2 \R^2 =  \varphi^2 \zeta_{\R}^2 \R^2.
\]
    But this contradicts the definition of $x_{t+1}'$ for which $\dist(x_{t+1}', \xast)^2 = \varphi^2 \zeta_{\R}^2 \R^2$. Thus $\eta_\ast \geq 1/\L$, and the inductive statement holds.

    For the statement about the hyperbolic space, it is enough to show it for $\kmin = \kmax = -1$, since the other cases can be reduced to this one by rescaling, see \citep[Remark 24]{martinez2020global}. We prove $\dist(x_t, x^\ast) \leq \varphi \R$ by induction in a similar way. It holds for $t=0$ by definition and notice that the proof above also works for this case until \eqref{eq:concluding_rgd_iterate_boundedness} if we now consider $\eta_\ast \defi \max_{\eta > 0} \{ \eta \ | \ \expon{x_t}(-\eta_\ast \nabla f(x_t)) \in \ball(\xast, \varphi \R)\}$, which yields $\dist(x_{t+1}', x^\ast) = \phi \R$, we make use of $\zeta \defi \varphi \zeta_{\R}$, and use $\L$-smoothness in $\ball(\xast, \varphi \R)$. However, we substitute the right hand side of $\circled{1}$ in \cref{eq:concluding_rgd_iterate_boundedness} by $\varphi^2\L\R^2/2$ which holds since by \citep[Proposition 13]{criscitiello2023curvature} we have $\Delta_0 \leq \frac{\varphi\L\R^2}{2\zeta_{\R}}$ for the hyperbolic space, and using $\zeta \leq \varphi \R + 1$ and $\zeta_{\R} \geq \R$, we bound $(\zeta-1) / \zeta_{\R} \leq \varphi$, and use $\varphi+1 = \varphi^2$. We note that \citep[Proposition 13]{criscitiello2023curvature} states global g-convexity as an assumption, but the proof only uses $f \in \smc$. Now we conclude as before.
Using $\eta_\ast \L \in (0, 1)$ dropping all the terms with $\Delta_i, \Delta_{t+1}' \geq 0$, and simplifying, we obtain $\dist(x_{t+1}', \xast)^2 < \varphi^2 \R^2$. This contradicts the definition of $x_{t+1}'$ for which $\dist(x_{t+1}', \xast)^2 = \varphi^2 \R^2$. Consequently, $\eta_\ast \geq 1/\L$, and we showed the inductive statement.
\end{proof}

We note that if we were to assume smoothness in the closed ball $\ball(\xast, 2\varphi \R \zeta_{\R})$, we could just write \cref{eq:rgd-descent} for all $i = 0$ to $t$, by arguing that $x_{t+1}$ is in such ball since $\dist(x_{t+1}, \xast) \leq \dist(x_{t}, \xast) + \dist(x_{t+1}, x_t)  \leq \varphi \R \zeta_{\R} + \varphi \R \zeta_{\R}$, where the bound on the first term is due to the induction hypothesis and the one of the second term is due to the definition of $x_{t+1}$ and that $\L$-smoothness implies $\frac{1}{\L}\norm{\nabla f(x_t)} \leq \dist(x_t, \xast)$. In this case, the proof proceeds in a similar but simpler way, without having to argue by contradiction or having to talk about the last iterate using a different learning rate. But the proof we presented requires smoothness only in a smaller region.

\convrateRGDL*

\begin{proof}\linkofproof{corol:conv_rate_RGD_1_L}
By \cref{lemma:iterate_boundedness_RGD_1_L}, in the general case our iterates stay in $\X$, that is, $\Rmax \leq \varphi \R \zeta_{\R}$. Note $\zeta_{\varphi \R \zeta_{\R}} = \bigo{\zeta_{\R}^2}$.
    For the g-convex case, the corollary is an immediate consequence of this iterate bound and of the convergence result in \citep[Theorem 13]{zhang2016first}. This theorem proves rates $\bigo{\zeta_{\Rmax}\frac{\L\R^2}{\epsilon}}$, which by using the bound on $\Rmax$ yields $\bigo{\zeta_{\R}^2 \frac{\L \R^2}{\epsilon}}$. Similarly, if we apply the \RGD{} result in \cite{martinez2023acceleratedmin} that has rates $\bigo{\frac{\L\Rmax[2]}{\epsilon}}$ for a function in $\smc$, we obtain the convergence rate $\bigo{\zeta_{\R}^2 \frac{\L \R^2}{\epsilon}}$. Note that the two approaches give the same rates.

    For the $\mu$-strongly g-convex case, the corollary is an immediate consequence of our iterate bound and both the result by \cite{udriste1994convex} and the one by \citep[Proposition 17]{martinez2023acceleratedmin} with $\X$ being the ball \cref{lemma:iterate_boundedness_RGD_1_L}, so that the algorithm becomes \RGD{}. Both have rates of $\bigo{\frac{\L}{\mu}\log(\frac{f(\xinit)-f(\xast)}{\epsilon})}$. The result is derived by the bound $f(\xinit)-f(\xast) \leq \L\R^2/2$ due to smoothness. Note that due to our iterate bounds, we just need to assume $\L$-smoothness and the $\mu$-strong g-convex in $\X=\ball(\xast, \R\zeta_{\R}(1+\sqrt{5})/2)$ and without these bounds, this rate is not obtained for the previous results for \RGD{} since otherwise we cannot even specify where the $\L$-smoothness and the $\mu$-strong g-convex properties hold and we cannot necessarily take the value of some global properties since for many manifolds there is no globally smooth strongly g-convex function with finite condition number, namely all Hadamard manifolds for which $\kmax < 0$ \cite{criscitiello2022negative}.

    We note that for the strongly g-convex case, we could also use the result in \citep[Theorem 15]{zhang2016first} yielding the potentially worse rate $\bigo{(\zeta_{\Rmax}+\frac{\L}{\mu})\log(\frac{\L\R^2}{\epsilon})} = \bigo{(\zeta_{\R}^2+\frac{\L}{\mu})\log(\frac{\L\R^2}{\epsilon})}$.

    We now show the statement for the hyperbolic space.
 By \citep[Theorem 16]{martinez2023acceleratedmin}, we have that for $\L$-smooth and g-convex functions, the complexity of \RGD{} is $\bigo{\frac{\L \Rmax^2}{\varepsilon}}$, where $\Rmax=\max_{t\in [T]}\dist(x_t,x^{*})$ is the largest distance to the optimizer.
    By \cref{lemma:iterate_boundedness_RGD_1_L}, we have that the iterates of \RGD{} stay in $\ball{\xast,\varphi \R}$, hence $\Rmax\le \varphi \R$ and the complexity becomes $\bigo{\frac{\L\R^2}{\varepsilon}}$. Note that in this case, the rate in \citep[Theorem 13]{zhang2016first} is seemingly worse according to their stated convergence rate. However, their proof reveals that their rates are actually $\bigo{\frac{(\zeta-1)\Delta_0+LR^2}{\epsilon}}$, which using the fact, as in the proof of \cref{lemma:iterate_boundedness_RGD_1_L}, that $\Delta_0 = \bigo{\frac{\L\R^2}{\zeta}}$, we obtain rates the same rates of convergence $\bigo{\frac{LR^2}{\epsilon}}$.
\end{proof}

\begin{remark}\label{remark:level_sets}
    We note that \RGD{} with step size $\eta < 2/\L$ does not increase the function value. Indeed, assume smoothness holds between $x_t$ and $x_{t+1} \defi \expon{x_t}(-\eta\nabla f(x_t))$. We have:
    \[
        f(x_{t+1}) \leq f(x_t) + \innp{\nabla f(x_t), \exponinv{x_t}(x_{t+1})} + \frac{\L}{2}\dist(x_t, x_{t+1})^2 = f(x_t) + \left(-\eta + \L \eta^2\right) \norm{\nabla f(x_t)}^2 \circled{1}[<] f(x_t),
    \]
    where $\circled{1}$ is equivalent to $\eta < 2/\L$. This means that if $\Y \defi \{y\in\M \ |\ f(y) \leq f(x_0)\}$ is the level set of $f$ with respect to $x_0$, the iterates of \RGD{} in the setting of \cref{lemma:iterate_boundedness_RGD_1_L} stay in $\X \cap \Y$ and we only need to assume g-convexity and $\L$-smoothness in that set. Note that if $f$ satisfies these properties in $\Y$, one can also obtain a convergence result by using $\diam(\Y)$ as a bound for $\Rmax$. And if $f$ is $\mu$-strongly g-convex in $\Y$ we obtain the rate $\bigotilde{\L/\mu}$. We can compute a bound that suggests that in general, the level set could have points that are $\R\sqrt{\L/\mu}$ away from the minimizer. Indeed, let $y \in \Y$, we obtain
    \[
        \frac{\mu }{2}\dist(y, x)^2  \circled{1}[\leq] f(y) - f(\xast) \circled{2}[\leq] f(\xinit) - f(\xast) \circled{3}[\leq] \frac{\L}{2}\dist(\xinit, \xast)^2,
    \]
    where above $\circled{1}$ uses $\mu$-strong g-convexity, $\circled{2}$ uses $y \in \Y$ and  $\circled{3}$ uses $\L$-smoothness.  This bound can be much larger than the $\bigo{\R\zeta_{\R}}$ bound in \cref{lemma:iterate_boundedness_RGD_1_L}.
\end{remark}

\begin{remark}\label{remark:reduce_to_str_cvx_for_RGD}
    We note that if we know a valid upper bound $\R$ on the initial distance, we can instead minimize $F(x) \defi f(x) + r(x)$ where $r(x) \defi \frac{\epsilon}{2\R^2}\dist(\xinit, x)^2$. In that case, a point $\hat{x}$ that is an $(\epsilon/2)$-minimizer of $r$ will be an $\epsilon$ minimizer of $f$, since
\[
    f(\hat{x}) \leq f(\hat{x}) + r(\hat{x}) \leq f(\xast) + r(\xast) + \frac{\epsilon}{2} \leq f(\xast) + \epsilon.
\]
    Denote $\hat{x}^\ast \defi \argmin_x r(x)$. We have by \cref{lemma:prox_is_frejer} that $\dist(\hat{x}^\ast, \xinit) \leq \dist(\xast, \xinit) \leq \R$. The smoothness constant of $F(\cdot)$ in $\X \defi \ball(\hat{x}^\ast, \varphi \R\zeta_{\R}) \subseteq \ball(\xinit, \varphi \R\zeta_{\R}+\R)$ is $\hat{L} + \zeta_{\varphi \R \zeta_{\R} + \R}\frac{\epsilon}{\R^2} = \bigo{\hat{L}+\zeta_{\R}^2\frac{\epsilon}{\R^2}}$, where $\hat{L}$ is the smoothness constant of $f$ in $\X \subseteq \ball(\xinit, \varphi \R\zeta_{\R}+\R)$. The strong convexity constant of $F(\cdot)$ in $\X$ is at least $\frac{\epsilon}{\R^{2}}\delta_{\varphi \R \zeta_{\R} + \R }$. The computation of these smoothness and strong convexity constants comes from \cref{fact:hessian_of_riemannian_squared_distance}. We know by \cref{lemma:iterate_boundedness_RGD_1_L} that the iterates of $\RGD$ on $F(\cdot)$ with step size $\eta \defi (\hat{L} + \zeta_{\varphi \R \zeta_{\R} + \R})^{-1}$ stay in $\X$ and thus by \cref{corol:conv_rate_RGD_1_L} we obtain the convergence rate  $\bigotilde{\frac{\zeta_{\R}^2}{\delta_{\varphi \R \zeta_{\R} + \R }} + \frac{\hat{L}\R^2}{\epsilon\delta_{\varphi \R \zeta_{\R} + \R }}}$. If we consider for instance a Hadamard manifold, that satisfies $\delta_r = 1$ for any $r$, we obtain the convergence rate $\bigotilde{\zeta_{\R}^2 + \frac{\hat{L}\R^2}{\epsilon}}$. We note that we can reduce the logarithmic factor if we use the reduction in \cite{martinez2020global}.

    In the hyperbolic space $\mathbb{H}$ of curvature $-1$, we make use of \citep[Proposition 13]{criscitiello2023curvature}, which says that for a differentiable g-convex $\L$-smooth function $f :\mathbb{H} \to \mathbb{R}$ in $\ball(x_0, \R)$ with a global minimizer $\xast$ in that ball, it is $f(x_0) -f(\xast)\leq 4L\dist(x_0, \xast)^2 / \zeta_{\dist(x_0, \xast)}$. With this proposition, we conclude that without loss of generality, we work with $\epsilon \leq \L\R^2/\zeta_{\R}$. On the other hand, using \cref{lemma:iterate_boundedness_RGD_1_L} and the reasoning above, we conclude that with $\hat{L}$ being the smoothness constant of $f$ in $\X \defi \ball(\hat{x}^\ast, \varphi \R)$, we obtain a convergence rate of $\bigotilde{ \zeta_{\R} + \frac{\hat{L}\R^2}{\epsilon}} = \bigotilde{\frac{\hat{L}\R^2}{\epsilon}}$. The last expression holds by our remark on the value of $\epsilon$.
\end{remark}

Now we prove our results for \RGD{} with a different step size. The proof of iterate boundedness is similar to our proof of \cref{lemma:iterate_boundedness_RGD_1_L}.

\rgdzeta*

\begin{proof}\linkofproof{lem:rgd-zeta}
    Define $\zeta \defi \zeta_{\R\sqrt{3/2}}$.
    We begin by showing that the choice of $\eta$ suffices to ensure that $\dist(x_t,\xast)\le \sqrt{3/2}\R$ for all $t\ge 0$ via induction.
  By assumption, the statement holds for $t=0$.
  Suppose the property holds for $t$, then we show that it also holds for $t+1$.
  By the induction statement, we have that $x_t\in \ball(\xast,\sqrt{3/2}\R)$.
  Let $\eta_\ast \defi \max_{\eta > 0} \{ \eta \ | \ \expon{x_t}(-\eta_\ast \nabla f(x_t)) \in \ball(\xast, \sqrt{3/2}\R)\}$.
  Note we can write a maximum because the ball is compact.
  It is enough to show $\eta_\ast \geq \frac{1}{\zeta \L}$.
  Suppose for the sake of contradiction that $\eta_\ast < \frac{1}{\zeta \L}$ and let $x_{t+1}' \defi \exponinv{x_t}(-\eta_\ast \nabla f(x_t))$.
  By definition, it must be $\dist(x_{t+1}', \xast) = \sqrt{3/2}\R$.
  We use the shorthands $\Delta_i\defi f\left(x_i\right)-f\left(x^*\right)$ and $g_i\defi \nabla f(x_i)$.
By $\L$-smoothness of $f$ in $\X$, we have for all $i < t$:
\begin{equation}
  \label{eq:rgd-descent_smaller_lr}
    \Delta_{i+1}-\Delta_i \leq \langle g_i, \exponinv{x_i}(x_{i+1})\rangle+\frac{\L}{2} \dist\left(x_{i+1}, x_i\right)^2=\left(-\frac{1}{\zeta \L}+ \frac{1}{2\zeta^2\L}\right)\Vert g_i\Vert^2 = -\frac{2\zeta -1}{2\zeta^2 \L}\Vert g_i\Vert^2.
\end{equation}
Similarly, we have
\begin{equation}
  \label{eq:rgd-descent_smaller_lr_t}
    \Delta'_{t+1}-\Delta_t \leq \langle g_i, \exponinv{x_t}(x'_{t+1})\rangle+\frac{\L}{2} \dist\left(x_{t+1}, x_t\right)^2=\left(-\eta_{*}+ \frac{\eta_{*}^2\L}{2}\right)\Vert g_t\Vert^2 \circled{1}[<]-\eta_{*}\frac{2\zeta-1}{2\zeta}\Vert g_t\Vert^2.
\end{equation}
    where $\circled{1}$ is equivalent to $\eta_{*}\in (0,1/(\L\zeta))$. By the g-convexity of $f$, the cosine inequality \cref{lemma:cosine_law_riemannian} and $\zeta_{\dist(x_i,\xast)}\le \zeta$ which holds by the induction hypothesis and the monotonicity of $r\mapsto \zeta_r$ we obtain $\circled{1}$ below for $i<t$
\begin{equation}
  \label{eq:rgd-cvx_smaller_lr}
\begin{aligned}
  \Delta_i \leq\left\langle-g_i, \exponinv{x_i}(x^*)\right\rangle &\circled{1}[\leq] \frac{1}{2\eta}\left[\dist(x_i, x^*)^2-\dist(x_{i+1}, x^*)^2 + \zeta \dist(x_i,x_{i+1})^2\right]\\
  &\circled{2}[=]\frac{\L\zeta}{2}\left[\dist(x_i, x^*)^2-\dist(x_{i+1}, x^*)^2 \right] + \frac{1}{2\L}\Vert g_{i}\Vert^2,
\end{aligned}
\end{equation}
    where $\circled{2}$ follows by the definition of $x_{i+1}$ and $\eta$.
Similarly, we have
\begin{equation}
  \label{eq:rgd-cvx_smaller_lr_t}
 \Delta_t\le \frac{1}{2\eta_{*}} \left( \dist(x_t,\xast)^2- \dist(x_{t+1}',\xast)^2 \right) + \frac{\zeta\eta_{*}}{2}\norm{g_t}^2.
\end{equation}
    Multiplying \cref{eq:rgd-descent_smaller_lr} by $C\defi \zeta^2/(2\zeta-1)$ and adding it to \cref{eq:rgd-cvx_smaller_lr}, we have for $i<t$
\begin{equation} \label{eq:rgd-step_smaller_lr}
C \Delta_{i+1}-(C-1) \Delta_i \leq \frac{\L\zeta}{2}\left(\dist(x_i, x^*)^2-\dist(x_{i+1}, x^*)^2\right).
\end{equation}
And analogously multiplying \cref{eq:rgd-descent_smaller_lr_t} by $C$ and adding it to \cref{eq:rgd-cvx_smaller_lr_t} and multiplying by $\eta_\ast \L\zeta < 1$:
\begin{equation}
  \label{eq:rgd-step_smaller_lr_t}
    \eta_\ast \L\zeta (C \Delta'_{t+1}-(C-1) \Delta_t ) < \frac{\L\zeta}{2}\left(\dist(x_t, x^*)^2-\dist(x_{t+1}', x^*)^2\right).
\end{equation}
    Now summing \cref{eq:rgd-step_smaller_lr} from $i=0$ to $t-1$ and adding \cref{eq:rgd-step_smaller_lr_t}, we obtain $\circled{2}$ below:
\begin{align}\label{eq:rgd-conv_smaller_lr}
    \begin{aligned}
    0&\circled{1}[\le] \eta_\ast \L \zeta C \Delta'_{t+1}+ (1-\eta_\ast \L\zeta)(C-1)\Delta_t + \sum_{i=1}^{t} \Delta_i \\
        &\circled{2}[<] (C-1) \Delta_0+\frac{\L\zeta  \dist(\xinit, \xast)^2}{2} -\frac{\L\zeta  \dist(x_{t+1}', \xast)^2}{2}  \\
        &\circled{3}[\le] (C-1+\zeta)\frac{\L D^2}{2}-\frac{\L\zeta \dist(x_{t+1}',\xast)^2}{2}\\
        &\circled{4}[\le] \frac{3\zeta}{2}\cdot\frac{\L D^2}{2}-\frac{\L\zeta \dist(x_{t+1}',\xast)^2}{2}
    \end{aligned}
\end{align}
    where $\circled{1}$ holds since $\Delta_i, \Delta_{t+1}' \ge 0$, $\eta_\ast \L \zeta < 1$, and $\circled{3}$ follows from $\Delta_0\le \frac{\L D^2}{2}$, which is implied by the  $\L$-smoothness of $f$. Finally $\circled{4}$ can be shown since $C-1+\zeta$ is increasing for $\zeta \in [1, \infty)$ and the limit at $+\infty$ is $3/2$.
    In particular, it follows that $\dist(x'_{t+1},\xast)^2< (3/2)\R^2$, which contradicts the fact that $\dist(x_{t+1}',\xast)^2=(3/2)\R^2$. We note that for $\zeta = 1$, we would have $C-1+\zeta = 1$ and we could make the radius of the ball equal to $RD$. Hence, $\eta_{*}\ge \frac{1}{\L\zeta}$ and the inductive statement is proven.

Now, using $\eta=1/(\zeta \L)$ to define $x_{t+1}$, we have that \cref{eq:rgd-conv_smaller_lr} holds for $i = t$ as well. And adding \cref{eq:rgd-conv_smaller_lr} up from $i=0$ to $t$, we conclude
\begin{equation} \label{eq:rgd-conv_smaller_lr+}
    0\le (t+1)\Delta_{t+1}\circled{1}[\le] C \Delta_{t+1}+\sum_{i=1}^{t} \Delta_i \le \frac{3\zeta \L \R^2}{4}-\frac{\L \zeta }{2}\dist(x_{t+1},\xast)^2,
\end{equation}
where $\circled{1}$ holds since $C \geq 1$, and also $\Delta_{t+1}\le \Delta_i$ for $i \leq t$ which follows from \cref{eq:rgd-descent_smaller_lr}.
Now, \cref{eq:rgd-conv_smaller_lr+} with $t\gets (T-1)$ implies that $\Delta_T\le \frac{3\zeta \L \R^2}{4T}$.
Thus, we have that $\Delta_T \leq \epsilon$ for $T\geq \frac{3\zeta \L D^2}{4\epsilon}$.
    We now prove the result for the $\mu$-strongly g-convex case. The algorithm is the same, and thus the iterates stay in $\ball(\xast,\sqrt{3/2}\R)$. The guarantee we just showed for the g-convex case implies $\circled{2}$ below:
\[
    \frac{\mu}{2}\dist(x_T, \xast)^2 \circled{1}[\leq] f(x_T) -f(\xast) \circled{2}[\leq] \frac{3\zeta \L\dist(\xinit, \xast)^2}{4T} \circled{3}[\leq]  \frac{\mu}{4}\dist(\xinit, \xast)^2
\]
    where $\circled{1}$ holds by $\mu$-strong g-convexity and $\circled{3}$ holds if $ T=\lceil 3\zeta\frac{\L}{\mu} \rceil$. Consequently, after $\bigo{\zeta\frac{\L}{\mu}}$ iterations we reduce the distance squared to the minimizer by a factor of $2$. Applying the same argument again sequentially for $r \defi \lceil \log_2 (\frac{\L D^2}{\epsilon}) \rceil$ stages of length $T$, we obtain that after $\hat{T} \defi rT  = \bigo{\zeta\frac{\L}{\mu}\log(\frac{\L D^2}{\epsilon})}$, iterations we have
    \[
        f(x_{\hat{T}}) - f(\xast) \leq \L\dist(x_{\hat{T}}, \xast)^2 \leq \frac{\L\dist(\xinit, \xast)^2}{2^{r}} \leq \epsilon.
    \]
\end{proof}

\nonsmoothrgd*

\begin{proof}\linkofproof{lem:rgd-nonsmooth}
    Denote by $v_i \in \partial f(x_i)$ the subgradients obtained and used by the algorithm. By the Lipschitzness assumption, it is $\norm{v_i} \leq \Lips$.
  We show by induction that $\dist(x_t,\xast)\le \R$ for all $t\ge 0$.
    For $t=0$ the statement holds by definition.  Assume that the statement holds for $t \leq T-1$, then we show that is also holds for $t+1$.
  By g-convexity of $f$, we have for $i\le t$
  \begin{align*}
      f(x_i)-f(\xast)&\le \langle v_i,-\exponinv{x_i}(\xast)\rangle = \frac{1}{\eta}\innp{-\exponinv{x_i}(x_{i+1}) ,-\exponinv{x_i}(\xast)} \\
      &\circled{1}[\le] \frac{1}{2\eta} \left[ \zeta_{\R} \dist(x_i,x_{i+1})^2+\dist(x_i,\xast)^2-\dist(x_{i+1},\xast)^2 \right]\\
      &\circled{2}[\leq] \frac{1}{2\eta} \left[ \dist(x_i,\xast)^2-\dist(x_{i+1},\xast)^2 \right]+ \frac{\zeta_{\R} \Lips[2]\eta}{2},
  \end{align*}
  where $\circled{1}$ holds by the cosine inequality  \cref{remark:tighter_cosine_inequality} and the monotonicity of $\R\mapsto \zeta_{\R}$.
  Further, $\circled{2}$ uses the definition of $x_{i+1}$ and Lipschitzness of $f$ in $\X$.
  Summing up the previous equation from $i=0$ to $t$, using $\dist(\xinit, \xast)=\R\leq \sqrt{2}\R$, $t \leq T-1$, and  $\eta=\frac{\R}{\Lips \sqrt{\zeta_{\R} T}}$ yields
\begin{equation}\label{eq:rgd-nonsmooth}
    0\le  2\eta\sum_{i=0}^{t}[f(x_i)-f(\xast)]\le \R^2 -\dist(x_{t+1},\xast)^2 + \zeta_{\R}(t+1)\Lips[2] \eta^2  \leq -\dist(x_{t+1},\xast)^2 +  2\R^2.
  \end{equation}
This proves the induction statement, since $\dist(x_{t+1},\xast)\le \sqrt{2}\R$.
From \cref{eq:rgd-nonsmooth} with $t\gets T-1$ and dropping the negative distance term, we obtain
\begin{equation*}
 \frac{1}{T}\sum_{t=0}^{T-1}[f(x_t)-f(\xast)]\le \frac{\R^2}{\eta T} = \frac{\sqrt{\zeta_{\R}}\Lips \R}{ \sqrt{T}}.
\end{equation*}
    Lastly, note that geodesic average of $\{x_{0},\ldots,x_{T-1}\}$ denoted by $\bar{x}_{T-1}$ as defined by \cref{eq:g-avg-1} satisfies $f(\bar{x}_{T-1})\le \frac{1}{T}\sum_{t=0}^{T-1}f(x_t)$ by \cref{cor:g-avg-1}.
\end{proof}

\section[RPPA proofs]{\RPPA{} proofs}

\exactRPPAvi*
\begin{proof}\linkofproof{prop:exact-rppa}
  Let $F(z)=\partial f(z)$ or $F(z)=(\partial_xf(x,y),-\partial_yf(x,y))$ with $z=(x,y)$. Recall that in the latter case, $\dist(z_1,z_2)= \sqrt{\dist(x_1,x_2)^2+\dist(y_1,y_2)^2}$.
 Since $f$ is g-convex or g-convex-concave, we have that
  \begin{equation}
    \label{eq:cc}
  0\le \innp{-v(z),\exponinv{z}{z^{*}}},\quad \forall z\in\X,
\end{equation}
where $v(z)\in F(z)$.
    For $t=0$, $\zinit \in \ball(\zast, \R)$ by definition. Fix $t \geq 0$ and assume $\dist(z_t, \zast) \leq \R$, then we are done if we show $\dist(z_{t+1}, \zast) \leq \dist(z_{t}, \zast) \leq \R$. By \cref{lemma:prox_is_frejer}, it is $\dist(z_t, z_{t+1}) \leq \dist(z_{t}, \zast) \leq \R$. By the triangular inequality, we have that the diameter of the geodesic triangle
$\triangle z_{t+1}z_t \zast\le 2\R$. Let $v_{t+1}\in F(z_{t+1})$, then this fact along with the monotonicity $r \mapsto \delta_r$ and the cosine inequality \cref{lemma:cosine_law_riemannian} implies $\circled{2}$ below
  \begin{equation}
  \begin{aligned}
      0 & \circled{1}[\leq]\innp{-v_{t+1},\exponinv{z_{t+1}}(\zast)} \\
      & \circled{2}[\leq]\frac{1}{\eta}\innp{-\exponinv{z_{t+1}}(z_t),\exponinv{z_{t+1}}(\zast)} \\
      & \circled{3}[\leq] -\frac{\delta_{2\R}}{2}\dist(z_{t+1}, z_t)^2 - \frac{1}{2} \dist(z_{t+1}, \zast)^2 + \frac{1}{2} \dist(z_t, \zast)^2 \\
      & \circled{4}[\leq] - \frac{1}{2} \dist(z_{t+1}, \zast)^2 + \frac{1}{2} \dist(z_t, \zast)^2,
  \end{aligned}
  \end{equation}
    where $\circled{1}$ holds by \cref{eq:cc}, $\circled{2}$ holds by since we have that $-\frac{1}{\eta} \exponinv{z_{t+1}}(z_t) = v_{t+1}$. In $\circled{3}$ we use \cref{lemma:cosine_law_riemannian} and in $\circled{4}$, we drop one negative term. The conclusion from this inequality is what we desired to prove.
\end{proof}

\inexactRPPAvi*
\begin{proof}\linkofproof{thm:rippa}
  Let $F(z)=\partial f(z)$ or $F(z)=(\partial_xf(x,y),-\partial_yf(x,y))$ with $z=(x,y)$. Recall that in the latter case, $\dist(z_1,z_2)= \sqrt{\dist(x_1,x_2)^2+\dist(y_1,y_2)^2}$.
 Since $f$ is $\mu$-strongly g-convex or $\mu$-strongly g-convex-concave, we have that
  \begin{equation}
    \label{eq:scsc}
  \mu \dist(z,z^{*})^2\le \innp{-v(z),\exponinv{z}{z^{*}}},\quad \forall z\in\X,
\end{equation}
where $v(z)\in F(z)$.
    We show by induction that $z_t \in B(\zast,\sqrt{2}\R)$ for all $t \geq 0$. For $t=0$ the property holds by definition. Now assume it holds for $t$, we will prove it for $t+1$. We first show that $\dist(z_{t+1},\zast)\le 3\R$.
  We have that
    \begin{equation}\label{eq:ppa-bound}
  \begin{aligned}
      \dist(z_{t+1},\zast)&\le \dist(z_{t+1},z^\ast_{t+1})+   \dist(\zast,z^\ast_{t+1})\circled{1}[\le]  \frac{1}{2} \dist(z_{t},z^\ast_{t+1}) + \dist(z_t,\zast) \circled{2}[\leq] \frac{3}{2} \dist(z_t, \zast) \circled{3}[\leq] 3\R.
  \end{aligned}
  \end{equation}
    where in $\circled{1}$ we used the criterion in Line \ref{line:ppa-criterion} and that by \cref{prop:exact-rppa} it is $\dist(z_{t+1}^{*},\zast)\le \dist(z_t,\zast)$. In $\circled{2}$ we used \cref{lemma:prox_is_frejer}, and we use the induction hypothesis in $\circled{3}$. We conclude that $\operatorname{diam}(\triangle z_{t+1}z_t\zast) \leq \dist(z_{t+1},\zast) + \dist(z_{t},\zast) \leq 5\R$. By $\mu$-strong g-convexity of $f$, with possibly $\mu =0$, $v_{t+1}\in F(z_{t+1})$ and the definition of $r_{t+1}$, we have
    \begin{equation}\label{eq:ppa-prog}
  \begin{aligned}
     0 &\le -\langle v_{t+1},\exponinv{z_{t+1}}(\zast)\rangle- \frac{\mu}{2}\dist(z_{t+1},\zast)^2\\
    &=\frac{1}{\eta}\langle-\exponinv{z_{t+1}}(z_t)-r_{t+1},\exponinv{z_{t+1}}(\zast)\rangle- \frac{\mu}{2}\dist(z_{t+1},\zast)^2 \\
      &\circled{1}[\leq] \frac{1}{\eta}\left( -\frac{\delta_{5\R}}{2}\dist(z_{t+1},z_t)^2-\frac{1}{2}\dist(z_{t+1},\zast)^2+\frac{1}{2}\dist(z_t,\zast)^2 \right) + \frac{1}{\eta}\left( \frac{\norm{r_{t+1}}^2}{2\Delta_{t+1}}+ \frac{\Delta_{t+1}}{2}\dist(z_{t+1},\zast)^2 \right) - \frac{\mu}{2}\dist(z_{t+1}, \zast)^2 \\
      &\circled{2}[\leq] -\frac{1}{2\eta}\dist(z_{t+1},\zast)^2+\frac{1}{2\eta}\dist(z_t,\zast)^2 + \frac{\Delta_{t+1}}{2\eta}\dist(z_{t+1},\zast)^2 - \frac{\mu}{2}\dist(z_{t+1}, \zast)^2 \\
    &= \frac{1}{2\eta}\left( (\Delta_{t+1}-1-\eta\mu)\dist(z_{t+1},\zast)^2+\dist(z_t,\zast)^2 \right).
  \end{aligned}
  \end{equation}
    where in $\circled{1}$, we used the cosine inequality \cref{lemma:cosine_law_riemannian} for the first term in the inner product. We also bounded the second term in the inner product by Young's inequality. In $\circled{2}$ we use the criterion in Line \ref{line:ppa-criterion} to bound $\norm{r_{t+1}}^2$ and cancel the result with the first summand.  We now separate two cases:
\\\\\textbf{Case \boldmath$\mu=0$.}
From \cref{eq:ppa-prog} with $\mu=0$ and $\Delta_{t+1}=(t+2)^{-2}$, we have that
  \begin{equation}\label{eq:ppa-nonexp}
      \dist(z_{t+1},\zast)^2\le (1-\Delta_{t+1})^{-1}\dist(z_t,\zast)^2\le \prod_{i=0}^t \frac{1}{1-\Delta_{i+1}}\dist(\zinit,\zast)^2\circled{1}[\le] 2\R^{2},
  \end{equation}
  where $\circled{1}$ holds since $\prod_{i=0}^{t} \frac{1}{1-(t+c)^{-2}}\le \frac{c}{c-1}$, by \cref{prop:bound-c}. This proves the induction statement.
    Note that changing the value of $\Delta_{t+1}$, we could have reduced the constant $2\R^2$ above to something as close to $\R^2$ as we want. For the convergence, we now sum \cref{eq:ppa-prog} from $t=0$ to $T-1$, divide by $T$ and use $\dist(\zinit, \zast) \leq \R$:
  \begin{equation*}
    \begin{aligned}
      \frac{1}{T}\sum_{t=0}^{T-1}\langle v_{t+1},\exponinv{z_{t+1}}(\zast)\rangle&\le \frac{1}{2\eta T}\left( \R^2-\dist(z_T,\zast)^2 + \sum_{t=0}^{T-1}\Delta_{t+1} \dist^{2}(z_{t+1},\zast) \right)\\
      &\circled{1}[\le]  \frac{1}{2\eta T}\left( \R^2 + 2\R^2 \sum_{t=0}^{T-1}\Delta_{t+1} \right) \circled{2}[\le]  \frac{3\R^2}{\eta T}
  \end{aligned}
  \end{equation*}
   In $\circled{1}$, we dropped a negative term and used \cref{eq:ppa-nonexp}. Then, $\circled{2}$ follows from $\sum_{t=0}^{T-1} \Delta_{t+1}\le\sum_{t=1}^{\infty} \frac{1}{t^2}\le \frac{\pi^2}{6}\le 2$.
   Note that for $f\in\proxc[\X]$, we have by g-convexity that
   \begin{equation*}
    f(x_{t+1})-f(x^{*})\le  \langle v_{t+1},\exponinv{z_{t+1}}(\zast)\rangle
  \end{equation*}
  and for $f\in \proxsp[\X]$, we have by g-convexity-concavity that
  \begin{equation*}
   f(x_{t+1},\yast)-f(\xast,y_{t+1})\le  \langle v_{t+1},\exponinv{z_{t+1}}(\zast)\rangle.
  \end{equation*}
Hence by using the uniform averaging scheme in \cref{cor:g-avg-1}, we obtain either $f(\bar{x}_T)\le \frac{1}{T}\sum_{t=0}^{T-1} f(x_{t+1})$ or $f(\bar{x}_T,\yast)-f(\xast,\bar{y}_T)\le \frac{1}{T}\sum_{t=0}^{T-1} f(x_{t+1},\yast)-f(\xast,y_{t+1})$, which concludes the proof.
\\\\\textbf{Case \boldmath$\mu>0$.}
By the definition of $\eta$ and $\Delta_{t+1}$, \cref{eq:ppa-prog} implies
\begin{equation}
  \label{eq:ppa-sc}
 \dist(z_{t+1},\zast)^2 \le \frac{\dist(z_t,\zast)^2}{1+\eta\mu/2}\le \ldots\le   (1+\eta\mu/2)^{-(t+1)}\R^2,
\end{equation}
    Since $\eta > 0$, we have that $0<\frac{1}{1+\eta\mu/2}< 1$ and hence $\dist(z_{t+1},\zast)\le \R$, which proves the induction statement. Now, if we set $T \geq (1+2/(\mu\eta))\ln(\frac{\R^2}{\varepsilon_d}) = \bigotilde{1+\frac{1}{\mu\eta}}$, we have that \cref{eq:ppa-sc} with $t\gets T-1$ implies
\begin{equation*}
 \dist(z_T,\zast)^2 \le \left(1-\frac{1}{1+2/(\mu\eta)}\right)^{T}\R^2\le \R^2 \exp \left( \frac{-T}{1+2/(\mu\eta)} \right) \leq \varepsilon_d,
\end{equation*}
which concludes the proof.
\end{proof}

\subsection[RIPPA implementation via composite RGD or PRGD]{\RIPPA{} implementation via composite \RGD{} or \PRGD{}}\label{sec:rippa_implementation}

We start by showing that a particular implementation of composite \RGD{} enjoys linear convergence, and as a corollary, we obtain that the subroutine in Line~\ref{line:ppa-criterion} of \cref{alg:rippa} can be implemented by using only $\bigotilde{1}$ gradient oracle calls when applied to smooth g-convex optimization defined in Hadamard manifolds, and using $\bigo{\zeta_{\R}}$ iterations of \PRGD{}. We note that if $g$ is an indicator function, the composite \RGD{} algorithm below is not the same as \PRGD{} in general. That is, the resulting algorithm is a projected \RGD{} that does not use a metric-projection.

\rgdcomposite*
\begin{proof}\linkofproof{lemma:composite_rgd}
    We first note that the $\argmin$ in the update rule exists. 
    Since $g$ is proper, lower semicontinuous and g-convex in $\X$, we have that $\Y \defi \X \cap \operatorname{dom}(g)$ is non-empty, closed and if $x \in \Y$ and $v\in\partial g(x)$, we have that $\{ y \in \Y\ |\ \frac{\L}{4}\dist(x_t, y)^2 + \innp{v, \exponinv{x}(y)} \leq \frac{\L}{4} \dist(x_t, x)^2 \}$ is compact by strong convexity of $x \mapsto \dist(x_t, x)^2$. We also have that $\{y\in Y \ |\ \frac{\L}{4}\dist(x_t, y)^2 + \innp{\nabla f(x), \exponinv{x_t}(y)} \leq \frac{\L}{4}\dist(x_t, x)^2 + \innp{\nabla f(x), \exponinv{x_t}(x)} \}$ is compact. The union of these two compact sets is compact and if we consider $z$ not in this union, we have  $\circled{2}$ below
\begin{equation*}
    \begin{aligned}
        \innp{\nabla f(x_t), \exponinv{x_t}(z)} + \frac{\L}{2}\dist(x_t, z)^2 + g(z) &\circled{1}[\geq]  \innp{\nabla f(x_t), \exponinv{x_t}(z)} + \frac{\L}{2}\dist(x_t, z)^2 + g(x) + \innp{v, \exponinv{x}(z)} \\
        &\circled{2}[>] \innp{\nabla f(x_t), \exponinv{x_t}(x)} + \frac{\L}{2}\dist(x_t, x)^2 + g(x),
  \end{aligned}
  \end{equation*}
    where $\circled{1}$ uses $v\in \partial g(x)$. This means that the minimization problem can be constrained to this union only and since it is compact the $\argmin$ exists.

    Now we prove the convergence result. We have
\begin{align*}
 \begin{aligned}
     F(x_{t+1}) &\circled{1}[\leq] \min_{x\in\X}\left\{f(x_t) + \innp{\nabla f(x_t), x - x_t}_{x_t} + \frac{\L}{2}\dist(x, x_t)^2 + g(x) \right\} \\
     &\circled{2}[\leq] \min_{x\in\X} \left\{F(x) + \frac{\L}{2}\dist(x, x_t)^2 \right\} \\
     & \circled{3}[\leq] \min_{\alpha \in [0,1]} \left\{\alpha F(x^\ast) + (1-\alpha) F(x_t) + \frac{\L\alpha^2}{2}\dist(x^\ast, x_t)^2 \right\} \\
     & \circled{4}[\leq] \min_{\alpha \in [0,1]} \left\{F(x_t) - \alpha\left(1-\alpha\frac{\L}{\mu}\right)\left(F(x_t)-F(x^\ast)\right) \right\} \\
     &\circled{5}[=] F(x_t) - \frac{\mu}{4\L}(F(x_t)-F(x^\ast)).
 \end{aligned}
\end{align*}
    Above, $\circled{1}$ holds by smoothness and the update rule of the composite Riemannian gradient descent algorithm. The g-convexity of $f$ implies $\circled{2}$. Inequality $\circled{3}$ results from restricting the min to the geodesic segment between $x^\ast$ and $x_t$ so that ${x = \expon{x_t}(\alpha \exponinv{x_t}(x^\ast) + (1-\alpha)\exponinv{x_t}(x_t))}$. We also use the g-convexity of $F$. In $\circled{4}$, we used strong convexity of $F$ to bound $\frac{\mu}{2} \dist(x^\ast, x_t)^2 \leq F(x_t)-F(x^\ast)$. Finally, in $\circled{5}$ we substituted $\alpha$ by the value that minimizes the expression, which is $\mu/2\L$. The result follows by subtracting $F(x^\ast)$ to the inequality above and recursively applying the resulting inequality from $t=1$ to $T \geq \frac{\L}{4\mu}\log(\frac{F(\xinit)-F(x^\ast)}{\epsilon})$.
\end{proof}

We now show that for smooth functions, we can simplify the inexactness criterion of \RIPPA{}.

\begin{lemma}\label{lemma:inexactness_criterion_for_prox_smooth}
    Under the assumptions of \cref{thm:rippa}, suppose that in addition $\ball(\xast,3\R) \subset \M$, $f$ is $\L$-smooth in $\ball(\xast, 2\R)$ and let
  \begin{equation}\label{eq:C-smooth}
      C_t \defi \min \left\{1/4,\frac{\Delta_{t+1}\delta_{3\R}}{2(\eta \L + \zeta_{3\R})^2+2\Delta_{t+1}\delta_{3\R}}  \right\}.
  \end{equation}
  It is enough that we guarantee $\dist(x_{t+1}^\ast, x_{t+1})^2 \leq C_t \dist(x_{t+1}^\ast, x_t)^2$ in order to satisfy the inexactness criterion in Line~\ref{line:ppa-criterion} of \cref{alg:rippa} at iteration $t$.
\end{lemma}

\begin{proof}
    Due to the definition of $C_t$, we just need to show the first part of the criterion in Line~\ref{line:ppa-criterion} of \cref{alg:rippa}. Fix $t \geq 0$. Firstly, we have 
    \[
        \dist(x_t, x_{t+1}) \leq \dist(x_t, x_{t+1}^\ast) + \dist(x_{t+1}^\ast, x_{t+1}) \circled{1}[\leq] \frac{3}{2} \dist(x_t, \xast),
    \] 
    where in $\circled{1}$ we used $C_t \leq 1/4$ and the fact that by \cref{lemma:prox_is_frejer}, it is $\dist(x_t,x_{t+1}^{*})\le \dist(x_t,\xast)$. So the diameter of $\triangle x_{t+1} x_t x_{t+1}^\ast$, which bounded by $\frac{1}{2}( \dist(x_{t}, x_{t+1})+ \dist(x_{t+1}^\ast, x_{t+1}) + \dist(x_t, x_{t+1}^\ast) )$, is thus most $\frac{1}{2}(\frac{3}{2} + 1 + 1) \dist(x_t, \xast) \leq 2\dist(x_t, \xast)$. If the statement of \cref{lemma:inexactness_criterion_for_prox_smooth} holds from iteration $0$ to $t-1$, then by \cref{thm:rippa} we have $\dist(x_{t}, \xast) \leq 2\R$.
    Now let $h_t(z)\defi f(z)+ \frac{1}{2\eta}\dist(z,x_t)^2$ be the proximal function at step $t$ which is thus smooth in  $\triangle x_{t+1} x_t x_{t+1}^\ast$ with constant $\bar{\L} \defi \L+\zeta_{3\R}/\eta$, where the $3\R$ comes from $\max\{\dist(x_t, x_{t+1}), \dist(x_t, \xast)\}$ and the bound above. Note that by definition $r_{t+1}=\eta \nabla h_t(x_{t+1})$. Hence, we have
\begin{equation}\label{eq:aux:bounding_norm_r}
     \norm{r_{t+1}}^2=\eta^2\norm{\nabla h_t(x_{t+1})}^2 \circled{1}[\le] \bar{\L}^2\eta^2 \dist(x_{t+1},x^\ast_{t+1})^2\circled{2}[\le]  \frac{2C_t}{1-2C_t} (\eta \L + \zeta_{3\R})^{2}\dist(x_{t},x_{t+1})^2
\end{equation}
Where $\circled{1}$ is due to the $\bar{\L}$-smoothness of $h_t$ we just showed, and the fact $x^\ast_{t+1}\in \argmin_{z\in \mathcal{M}} h_t(z)$. Further, $\circled{2}$ follows by the inexactness criterion, i.e.,
 \begin{align*}
   \dist(x_{t+1},x_{t+1}^{*})^2&\le C_t \dist(x_t,x^\ast_{t+1})^2\le  2C_t (\dist(x_t,x_{t+1})^2+ \dist(x_{t+1},x^\ast_{t+1})^2)\\
\Leftrightarrow   \dist(x_{t+1},x_{t+1}^{*})^2&\le  \frac{2C_t}{1-2C_t} \dist(x_t,x_{t+1})^2.
 \end{align*}
    Note that we want to prove $\norm{r_{t+1}}^2\le \Delta_{t+1}\delta_{5\R/2}\dist(x_t,x_{t+1})^2$, so by \eqref{eq:aux:bounding_norm_r} it is enough that 
 \begin{equation}
   C_t \le \frac{\Delta_{t+1}\delta_{3\R}}{2(\eta \L +\bar{\zeta})^2+2\Delta_{t+1}\delta_{3\R}},
 \end{equation}
    as specified in \cref{eq:C-smooth}. Note that for simplicity, we used $\delta_{3\R}$ which is less than $\delta_{5\R/2}$.
\end{proof}

Finally, we can show that we can implement \RIPPA{} for smooth functions.

\implementationRIPPA*

\begin{proof}\linkofproof{prop:nearly_constant_subroutine}
    We show that we can implement the subroutine at iteration $t$, starting from $x_t$, assuming that it was successfully implemented in previous iterations and thus according to \cref{thm:rippa}, we have $x_t \in \ball(\xast, 2\R)$. The exact optimizer of the prox $x_{t+1}^\ast$ satisfies, by \cref{lemma:prox_is_frejer}, that $\dist(x_t, x_{t+1}^\ast) \leq \dist(x_t, x^\ast) \leq 2\R$. Thus $x_{t+1}^\ast \in \Xt \defi \ball(x_t, 2\R) \subset \ball(\xast, 4\R)$ and by assumption $f$ is $\L$-smooth in $\Xt$.

    We now use the composite Riemannian Gradient Descent in \cref{lemma:composite_rgd} with the $\L$-smooth function $f$ of the statement the set $\Xt$ defined above and with $g(x) = \frac{1}{2\eta} \dist(x_t, x)^2 =  \frac{\L}{2} \dist(x_t, x)^2$, which is strongly g-convex in $\Xt$ with parameter $\mu \defi \L/\delta_{2\R}$, cf. \cref{fact:hessian_of_riemannian_squared_distance}. If we use $T \geq  \frac{\L}{4\mu}\log(\frac{\L(1+\zeta_{2\R})}{\mu C}) = \bigotilde{\frac{1}{\delta_{2\R}}}$ iterations on $F = f+g$, we obtain
    \begin{align*}
 \begin{aligned}
     \frac{\mu}{2} \dist(z_T, x_{t+1}^\ast)^2 &\circled{1}[\leq] F(z_T) - F(x_{t+1}^\ast) \circled{2}[\leq] \exp(-T\mu/4\L) (F(x_t)-F(x_{t+1}^\ast)) \\
     &\circled{3}[\leq] \exp(-T\mu/4\L) \frac{\L(1+\zeta_{2\R})}{2}\dist(x_t, x_{t+1}^\ast)^2 \circled{4}[\leq] \frac{\mu C \dist(x_t, x_{t+1}^\ast)^2}{2}.
 \end{aligned}
\end{align*}
    Above, $\circled{1}$ holds by $\mu$-strong g-convexity, $\circled{2}$ holds by the convergence guarantees in \cref{lemma:composite_rgd}. The intial gap $\circled{3}$ holds by the fact that $\nabla F(x_{t+1}^\ast) = 0$ and by $L(1+\zeta_{2\R})$-smoothness of $F$, since the smoothness of $g$ is $\L\zeta_{2\R}$ by \cref{fact:hessian_of_riemannian_squared_distance}. Finally $\circled{4}$ holds by the definition of $T$. This inequality is the criterion in \cref{lemma:inexactness_criterion_for_prox_smooth}, so the statement for composite \RGD{} is proven.

    For Hadamard manifolds, \cite{martinez2023acceleratedminmax} showed convergence of \PRGD{} for functions in $\smsc$, and in particular if the global minimizer is in the feasible set $\Xt$, the rates become $\bigo{\zeta_{\diam(\Xt)}\frac{\L}{\mu}}$ . Thus, taking into account that in Hadamard $\delta_r = 1$ for all $r\geq 0$ and using \PRGD{} on $F$ which is smooth with constant $\hat{L} =\zeta_{\R}\L$, and the exact same argument as above except for the new joint smoothness constant and except for $\circled{3}$ and $\circled{4}$ in which we use these other convergence rates and $T = \bigotilde{ \zeta_{2\R}\frac{\L}{\mu}} = \bigotilde{\zeta_{\R}^2}$, we also arrive to the criterion in \cref{lemma:inexactness_criterion_for_prox_smooth}. Note that the constant $C$ in the \cref{lemma:inexactness_criterion_for_prox_smooth} is polynomial in problems parameters, such as $\zeta_{\R}$, $\frac{1}{\delta_{2\R}}$, so the logarithm is benign.

\end{proof}

Note that the for the optimization required in the first iterate of \CRGD{} for \cref{prop:nearly_constant_subroutine} is the optimization of a quadratic in $\Tansp{x_t}\M$ with a ball constraint and therefore it can be easily implemented. It is in fact, equivalent to the first step of \PRGD{} with possibly a different step-size.

\subsection{RIPPA implementation via \RGDA{}}
\label{sec:min-max-rippa}

\begin{definition}
  A point $(x^{*},y^{*})\in \M\times \N$ is a saddle point of $f\in \proxsp$ if
  \begin{equation*}
   f(x^{*},y)\le f(x^{*},y^{*})\le f(x,y^{*}) \quad, \forall x\in \M, y\in \N.
  \end{equation*}
If a point $(\tilde{x},\tilde{y})\in \M\times \N$ satisfies
  \begin{equation*}
   f(\tilde{x},y)-f(x,\tilde{y})\le \varepsilon \quad, \forall x\in \M, y\in \N,
 \end{equation*}
we call $(\tilde{x},\tilde{y})\in \M\times \N$ an $\varepsilon$-saddle point of $f$.
\end{definition}

\RippaRgdImplement*
\begin{proof}\linkofproof{prop:rippa-minmax-rgd}
  For each iteration $t\in [T]$ of RIPPA, we employ \RGDA{} to approximately solve the prox problem, i.e.,
  \begin{equation}\label{eq:minmax-prox}
    \min_{x\in \M}\max_{y\in\N} \left\{h_t(x,y)\defi  f(x,y) + \frac{1}{2\eta}\dist(x,x_t)^2-\frac{1}{2\eta}\dist(y,y_t)^2\right\}.
  \end{equation}
    Choosing $\eta=1/\L$, we denote by $\bar{L}_D=\L(1+\zeta_D)$ and $\bar{\mu}= \L$ the smoothness and the strong g-convexity-concavity constants of $h_t$ in $\ball((x_t,y_t),D)$ and by $\bar{\kappa}_D=1+\zeta_D$ its condition number, cf. \cref{fact:hessian_of_riemannian_squared_distance} and take into account that $\delta_D =1$ regardless of the value of $D$, since we are in a Hadamard manifold.
    The smoothness and strong g-convexity-concavity constants of $h_t$ depend on the size of the set we consider. In order to fully characterize the complexity of \RGDA{}, we need to find a bound on $D$, such that all iterates of \RGDA{} lie in $\ball((x_t,y_t),D)$.
    \paragraph{Bounding the \RGDA{} iterates}
    Denote by $(\tilde{x}_{\tau},\tilde{y}_{\tau})$ the $\tau$-th iterates of \RGDA{} at iteration $t$ of RIPPA.
    In the following we find $D$ such that $\dist(\tilde{x}_{\tau},x_t)^2+\dist(\tilde{y}_{\tau},y_t)^2\le D^2$.
    Denote by $(x_{t+1}^{*},y_{t+1}^{*})$ the exact solution to \cref{eq:minmax-prox}.
  In general, the upper bound on this $D$ is unknown for \RGDA{}. But we do know that $\norm{\nabla_yh_t(\tilde{x}_{\tau},\tilde{y}_{\tau})}^2+\norm{\nabla_xh_t(\tilde{x}_{\tau},\tilde{y}_{\tau})}^2$ decreases monotonically at each iteration, which allows us to to bound $D$.

Let $\bar{R}\defi \sqrt{2}\R$, we have that
  \begin{equation}
    \label{eq:rgd-dist}
\begin{aligned}
  \dist(\tilde{x}_{\tau},x_{t+1}^{*})^2+\dist(\tilde{y}_{\tau},y_{t+1}^{*})^2&\circled{1}[\le] \frac{1}{\bar{\mu}^2}(\norm{\nabla_yh_t(\tilde{x}_{\tau},\tilde{y}_{\tau})}^2+\norm{\nabla_xh_t(\tilde{x}_{\tau},\tilde{y}_{\tau})}^2 )\\
  &\circled{2}[\le] \frac{1}{\bar{\mu}^2}(\norm{\nabla_xh_t(x_t,y_t)}^2+\norm{\nabla_yh_t(x_t,y_t)}^2)\\
    &\circled{3}[\le] \frac{ \bar{L}_{\bar{R}}^2}{\bar{\mu}^2} (\dist(x_t,x_{t+1}^{*})^2+\dist(y_t,y_{t+1}^{*})^2)\\
    &\circled{4}[\le] \frac{ \bar{L}_{\bar{R}}^2}{\bar{\mu}^2} (\dist(x_t,x^{*})^2+\dist(y_t,y^{*})^2)=(\zeta_{\bar{R}}+1)^{2} (\dist(x_t,x^{*})^2+\dist(y_t,y^{*})^2)
\end{aligned}
  \end{equation}
    where $\circled{1}$ follows from \citep[Proposition 38]{martinez2023acceleratedmin} since $f$ is strongly g-convex-concave, $\circled{2}$ holds due to the monotone decrease of the gradient norm by RGDA \citep[Lemma 5]{cai2023curvature}, $\circled{3}$ holds by the smoothness inequality of $h_t$ between $(x_t,y_t)$ and $(x^{*}_{t+1},y^{*}_{t+1})$, the fact that $\dist((x_t, y_t),(x_{t+1}^{*},y_{t+1}^{*})) \le \bar{R}$ which holds by \cref{lemma:prox_is_frejer} and by \cref{thm:rippa} and $\circled{4}$ holds by \cref{lemma:prox_is_frejer}.
Hence for all $\tau$ and all iterations of RIPPA $t$, we have that
\begin{align*}
  \dist(x_t,\tilde{x}_{\tau})^2+ \dist(y_t,\tilde{y}_{\tau})^2&\le 2(\dist(x_t,x_{t+1}^{*})^2+\dist(y_t,y_{t+1}^{*})^2+ \dist(x_{t+1}^{*},\tilde{x}_{\tau})^2+ \dist(y_{t+1}^{*},\tilde{y}_{\tau})^2)\\
  &\circled{1}[\le] 2(1+(1+\zeta_{\bar{R}})^2)(\dist(x_t,x^{*})^2+\dist(y_t,y^{*})^2) \circled{2}[\le] 4 (1+(1+\zeta_{\bar{R}})^2) \R^2,
\end{align*}
    where $\circled{1}$ holds by \cref{lemma:prox_is_frejer} and \cref{eq:rgd-dist} and $\circled{2}$ holds by \cref{thm:rippa}, that is $z_t \in \ball(z^\ast, \sqrt{2}\R)$, as the previous iterates of RIPPA also satisfy the inexactness criterion.
It follows that $D^2=4 (1+(1+\zeta_{\bar{R}})^2) \R^2$ and hence $D=\bigo{\R\zeta_{\R}}$, which means the we have bounded the distance of the \RGDA{} iterates from their initialization $x_t,y_t$ for each RIPPA iteration.\\\\

Furthermore, we specify how far off the iterates of \RGDA{} move from the saddle point $(x^{*},y^{*})$ at every iteration of RIPPA.
We require this in order to know the total size of the set in which we assume $f$ to be smooth and g-convex-concave. We have that
      \begin{align*}
        \dist(\tilde{x}_{\tau},x^{*})^2+\dist(\tilde{y}_{\tau},y^{*})^2& \le 2 (\dist(\tilde{x}_{\tau},x_t)^2+\dist(x_t,x^{*})^2+\dist(\tilde{y}_{\tau},y_t)^2+\dist(y_t,y^{*})^2)\\
        &\le 4\R^2 +8 (1+(1+\zeta_{\R})^2)\R^2\leq 44 \R^2 \zeta_{\R}^2 = \bigo{\R^2\zeta_{\R}^2}.
      \end{align*}
Consequently, it suffices to assume that $f$ is $\L$-smooth in $\ball((x^{*},y^{*}),7\R\zeta_{\R})$.
\paragraph{Inner loop complexity}
First, note that the criterion $\dist(x_{t+1},x_{t+1}^{*})^2+\dist(y_{t+1},y_{t+1}^{*})^2\le \frac{1}{4}(\dist(x_{t},x_{t+1}^{*})^2+\dist(y_{t},y_{t+1}^{*})^2) $ is not required to prove \cref{thm:rippa} for Hadamard manifolds.
It ensures a bound on $\diam(\triangle (x_t,y_t)(x_{t+1},y_{t+1})(x^{*},y^{*}))$ a-priori, which is required in order to lower bound $\delta_r$ in \cref{eq:ppa-prog} but for Hadamard manifolds we have $\delta_r=1$.
Hence we only need to ensure that the remaining criterion is satisfied, i.e.,
 \begin{equation}
   \label{eq:min-max-crit}
 \norm{r_{t+1}^x}^2+\norm{r_{t+1}^y}^2\le \Delta_{t+1}\left( \dist(x_t,x_{t+1})^2+  \dist(y_t,y_{t+1})^2 \right),
\end{equation}
where $\rtpx\defi \eta v_{t+1}^x-\log_{x_{t+1}}(x_t)$, $\rtpy\defi\eta v_{t+1}^y-\log_{y_{t+1}}(y_t)$ and $v_{t+1}^x\in \partial_x f(x_{t+1},y_{t+1})$, $v_{t+1}^y\in \partial_y f(x_{t+1},y_{t+1})$.
We have that
\begin{equation}
  \label{eq:criterion}
  \begin{aligned}
      \norm{\nabla_xh_t(x_{t+1},y_{t+1})}^2+\norm{\nabla_yh_t(x_{t+1},y_{t+1})}^2 &\circled{1}[\le] \exp\left(-\frac{T}{\bar{\kappa}_D^2}\right) ( \norm{\nabla_xh_t(x_{t},y_{t})}^2+\norm{\nabla_yh_t(x_{t},y_{t})}^2)\\
 &\circled{2}[\le] \exp\left(-\frac{T}{\bar{\kappa}_D^2}\right) \bar{L}_{D}^{2} (\dist(x_t,x_{t+1}^{*})^2+\dist(y_t,y_{t+1}^{*})^2)\\
                                          &\le 2\exp\left(-\frac{T}{\bar{\kappa}_D^2}\right) \bar{L}_{D}^2 (\dist(x_t,x_{t+1})^2+\dist(y_t,y_{t+1})^2+\dist(x_{t+1},x_{t+1}^{*})^2+\dist(y_{t+1},y_{t+1}^{*}))\\
                                                                             &\circled{3}[\le] 2\exp\left(-\frac{T}{\bar{\kappa}_D^2}\right) \bar{L}_{D}^2 (\dist(x_t,x_{t+1})^2+\dist(y_t,y_{t+1})^2)\\
                                                                             &\quad+2\exp\left(-\frac{T}{\bar{\kappa}_D^2}\right) \bar{\kappa}_D^2(\norm{\nabla_xh_t(x_{t+1},y_{t+1})}^2+\norm{\nabla_yh_t(x_{t+1},y_{t+1})}^2)\\
 &\le \frac{2\bar{L}_{\bar{R}}^2 \exp\left(-\frac{T}{\bar{\kappa}_D^2}\right)}{1-2\bar{\kappa}_D^2\exp\left(-\frac{T}{\bar{\kappa}_D^2}\right)} (\dist(x_t,x_{t+1})^2+\dist(y_t,y_{t+1})^2),
  \end{aligned}
\end{equation}
    where $\circled{1}$ holds by \citep[Lemma 5]{cai2023curvature}, $\circled{2}$ holds by the smoothness of $h_t$ between $(x_t,y_t)$ and $(x_{t+1}^{*},y_{t+1}^{*})$, noting that both of those points lie in $\ball((x^{*},y^{*}),D)$. $\circled{3}$ holds by the $\bar{\mu}$-strong g-convexity-concavity of $h_t$ between $(x_{t+1},y_{t+1})$ and $(x_{t+1}^{*},y_{t+1}^{*})$ and since both of those points lie in $\ball((x^{*},y^{*}),D)$.
 By definition of $\rtpx$, $\rtpy$ and \cref{eq:min-max-crit} we have
 \begin{equation}
   \label{eq:1}
    \norm{\nabla_xh_t(x_{t+1},y_{t+1})}^2+\norm{\nabla_yh_t(x_{t+1},y_{t+1})}^2 =\frac{1}{\eta^2} \left(\norm{r_{t+1}^x}^2+\norm{r_{t+1}^y}^2\right)\le \frac{\Delta_{t+1}}{\eta^2}\left( \dist(x_t,x_{t+1})^2+  \dist(y_t,y_{t+1})^2 \right).
 \end{equation}
  Hence after at most $T$ iterations such that
  \begin{equation*}
    \frac{2\bar{L}_D^2 \exp(\frac{-T}{\bar{\kappa}_D^2})}{1-2 \bar{\kappa}^2_{D}\exp(\frac{T}{-\bar{\kappa}^2_D})} \le \Delta_{t+1}\delta \L^2 \Leftrightarrow T \ge \bar{\kappa}^2_D\log \left(  \frac{2\bar{L}_D^2+2\bar{\kappa}^2_D\Delta_{t+1} \L^{2}}{\Delta_{t+1} \L^2}\right)
  \end{equation*}
the error criterion is satisfied.

We have that $\bar{\kappa}_D= \bigo{\zeta_D}$ and since $D=\zeta_{\R}R$ and $\zeta_{\R}R=\bigo{\zeta_{\R}^2}$, the iteration complexity of \RGDA{} for each RIPPA iteration is $\bigotilde{\bar{\kappa}_D^2}=\bigotilde{\zeta_{\R}^4}$.
Note that while $T$ depends on $\R$, the error criterion can be checked without knowledge of $\R$.

\end{proof}
\section{Prox properties}

We start by showing that the Moreau envelope may not be convex when positive curvature is present. After, we give a proof for the smoothness of the Moreau envelope. After the proof, we provide some other alternative proofs that obtain a less general result. We include them because their techniques are very different and could be of independent interest.

\subsection{Moreau envelope may fail to be g-convex in positive curvature}

\begin{remark}\label{rem:non_g_convex_moreau_envelope}
    \citet[Section 6.1]{wang2023online} provided an example in which the prox is not a nonexpansive operator. In particular the example is for the projection to a small enough geodesic segment (smaller than half a maximum circle) in a subset of the $2$-dimensional sphere. They show that if the segment has endpoints $x, y$ and we define $\hat{x} = \expon{x}(rv_x)$, $\hat{y} = \expon{y}(rv_y)$, where $r$ is small enough and $v_x\in \Tansp{x}\M$ and $v_y \in \Tansp{y}\M$ are unit vectors normal to $\exponinv{x}(y)$ and $\exponinv{y}(x)$, respectively, then $\dist(\hat{x}, \hat{y}) < \dist(x, y)$, which shows that the projection operator (which coincides with the proximal operator for the indicator function of the geodesic segment) is not a nonexpansive operator. If we consider the Moreau envelope of this geodesic segment in a uniquely geodesic neighborhood of the segment that contains the geodesic of minimum distance between $\hat{x}$ and $\hat{y}$, we have that the function value of the Moreau envelope is equal to half the distance squared to the segment. We have that $\dist(x, \hat{x}) = \dist(y, \hat{y}) = r$ but it is straighforward that the distance of $z \defi \expon{x}(\frac{1}{2}\exponinv{x}(y))$ to the segment is greater than $r$. Therefore, the Moreau envelope is not g-convex in this case.
\end{remark}

\subsection{Smoothness of Moreau envelope}

\smoothnessofmoreauenvelope*

\begin{proof}\linkofproof{prop:smoothness_of_moreau_envelope}
    Recall that we define $\prox(x)  \defi \argmin_{z\in\Mmin}\{f(z) + \indicator{\X}(z) + \frac{1}{2\eta} \dist(x, z)^2\} \in \X$. The result is derived from the following.
\begin{align*}
\begin{aligned}
    M(y) &= \min_{z\in\Mmin}\{f(z) + \indicator{\X}(z) + \frac{1}{2\eta} \dist(y, z)^2\} \\ & \circled{1}[\leq] f(\prox(x)) + \frac{1}{2\eta} \dist(y, \prox(x))^2 \\
    & \circled{2}[=] M(x) - \frac{1}{2\eta} \dist(x, \prox(x))^2 + \frac{1}{2\eta} \dist(y, \prox(x))^2 \\
    & \circled{3}[\leq] M(x) - \innp{\exponinv{x}(\prox(x)), \exponinv{x}(y)} + \frac{\zeta_{\dist(x, \prox(x))}}{2\eta}\dist(y, x)^2 \\
    & \circled{4}[=] M(x) + \innp{\nabla M(x), \exponinv{x}(y)} + \frac{\zeta_{\dist(x, \prox(x))}}{2\eta}\dist(x, y)^2.
\end{aligned}
\end{align*}
        Above, we just substituted in $\circled{1}$ the variable in the $\min$ by $\prox(x)$ yielding a possibly greater value. In $\circled{2}$, we used the definition of $M(x)$ and of $\prox(x)$ and we used the cosine inequality \cref{remark:tighter_cosine_inequality} in $\circled{3}$. In $\circled{4}$, we used \cref{corol:grad_of_moreau_envelope}.
    Since $\prox(x) \in \X$, then if $x\in\X$ we have $\dist(x, \prox(x)) \leq D$ and so given the inequality above, we have that $M(x)$ is indeed $(\zeta_D/\eta)$-smooth in $\X$.
\end{proof}

We now include alternative less general proofs of the fact proven in \cref{prop:smoothness_of_moreau_envelope}. They are strictly worse, but the techniques used can be of independent interest.

\subsection{Alternative proofs}\label{sec:alternative_proofs}
We start by showing the non-expansivity of the proximal operator. After we finished our result, we discovered that this fact was already proven in \cite{jost1995convex,mayer1998gradient}. We still include our proof since it is very different and arguably simpler. A proof can also be found in the book \citep[Theorem 2.2.22]{bacak2014convex}.

\label{sec:prox-properties}
\begin{lemma}[Non-expansivity of the prox]\label{lemma:non_expansivity_of_the_prox}
Consider a  function $f\in \proxc[\H]$, where $\H$ is a Hadamard manifold, and let $x,y \in \H$, $x^+\defi\operatorname{prox}_f(x)$, $y^+\defi\operatorname{prox}_f(y)$. Then
      \[\dist(x^+,y^+)\le \dist(x,y).\]
    \end{lemma}

    \begin{proof}
        Let $h_p(x)\defi f(x)+\frac{1}{2\eta}\dist(x,p)^2$. Note that $h_p$ is $(1/\eta)$-strongly g-convex, since $f$ is g-convex. Define $x^+=\argmin_{z\in \H}h_x(z)$ and $y^+= \argmin_{z\in \H}h_y(z)$. We note that $\partial f(\cdot) + \frac{1}{\eta}\exponinv{y}(\cdot)  = \partial h_y(\cdot)$ and similarly for $h_x$. We choose a subgradient $g^f_{y^+} \in \partial f(y^+)$ and define subgradients $g^{h_x}_{x^+} \in \partial h_x(x^+)$, $g^{h_y}_{y^+} \defi g^f_{y^+} + \frac{1}{\eta}\exponinv{y^+}(y) \in \partial h_y(y^+)$ and $g^{h_x}_{y^+} \defi g^f_{y^+} + \frac{1}{\eta}\exponinv{y^+}(x) \in \partial h_x(y^+)$ so that
        \begin{equation}\label{eq:aux:well_chosen_subgrads}
            g^{h_x}_{y^+} - g^{h_y}_{y+} = \frac{1}{\eta} \exponinv{y^+}(x)- \frac{1}{\eta}\exponinv{y^+}(y).
        \end{equation}
        By \cref{lemma:first_order_optimality}, we have:
          \[
          0\le    \langle g^{h_x}_{x^+},\exponinv{x^+}(z)\rangle, \forall z\in \H \quad \quad \text{ and } \quad \quad 0\le    \langle g^{h_y}_{y^+},\exponinv{y^+}(z')\rangle, \forall z'\in \H.
          \]
      Choosing $z=y^+$, $z'=x^+$, adding up and using Gauss lemma to transport to $y^+$, we obtain
      \begin{equation}
       \label{eq:22}
          0 \leq \langle g^{h_y}_{y^+}-\Gamma{x^+}{y^+}g^{h_x}_{x^+}, \exponinv{y^+}(x^+)\rangle.
     \end{equation}
        Furthermore, by the $(1/\eta)$-strong g-convexity of $h_x$, we have
    \begin{align} \label{eq:23}
    \begin{aligned}
        \frac{1}{2\eta}\dist(x^+,y^+)^2+ \frac{1}{2\eta}\dist(x^+,y^+)^2 &\leq \left( h_x(x^+) - h_x(y^+) - \innp{g^{h_x}_{y^+}, \exponinv{y^+}(x^+)}_{y^+}\right) \\
         & \quad + \left( h_x(y^+) - h_x(x^+) - \innp{g^{h_x}_{x^+}, \exponinv{x^+}(y^+)}_{x^+}\right) \\
         &= \innp{ \Gamma{x^+}{y^{+}}g^{h_x}_{x^+}- g^{h_x}_{y^+}, \exponinv{y^+}(x^+)}.
    \end{aligned}
    \end{align}
        Summing up \cref{eq:22,eq:23}, we get $\circled{1}$ below
     \begin{align*}
         \frac{1}{\eta}\dist(x^+,y^+)^2&\circled{1}[\le]  \langle  g^{h_y}_{y^+}- g^{h_x}_{y^+}, \exponinv{y^+}(x^+)\rangle \le \norm{g^{h_y}_{y^+}- g^{h_x}_{y^+}}\dist(x^+,y^+)\\
                                   & \circled{2}[=] \frac{1}{\eta}\norm{\exponinv{y^+}(x)-\exponinv{y^+}(y)}\dist(x^+,y^+)  \circled{3}[\leq]\frac{1}{\eta} \dist(x,y) \dist(x^+, y^+).
     \end{align*}
        Therefore $\dist(x^+,y^+)\le \dist(x,y)$. We used \eqref{eq:aux:well_chosen_subgrads} in $\circled{2}$. In $\circled{3}$ we use \cref{lemma:underestimation_of_Hadamard_distance} which holds for Hadamard manifolds.
    \end{proof}

Now we can give a slightly worse proof of the smoothness of the Moreau envelope by using \cref{lemma:non_expansivity_of_the_prox}.

\begin{proposition}
    Consider a $\mathcal{X} \subset \mathcal{H}$ closed and g-convex set where $\H$ is Hadamard Manifold. Let $f: \mathcal{H}\rightarrow \mathbb{R}$ be a g-convex function in $\XX$. Then the gradient of the Moreau envelope $M(x)\defi \min_{z\in \H} \{ f(z) + \indicator{\X}(z)+\frac{1}{2\eta} \dist(x,z)^2\}$ is Lipschitz with constant $\L \defi\frac{1+\zeta}{\eta} = \bigo{\zeta/\eta}$, i.e.,
  \begin{equation*}
      \norm{\nabla M(x)- \Gamma{y}{x} \nabla M(y)}_x\le \L \dist(x,y).
  \end{equation*}
\end{proposition}

\begin{proof}
    Let $x^+\defi\text{prox}(x)=\argmin_{z\in \M} \{f(z)+\indicator{X}(z)+\frac{1}{2\eta}\dist(x,z)^2\}$. The following holds:
 \begin{align*}
     \norm{\nabla M(x)- \Gamma{y}{x} \nabla M(y)}_x&\circled{1}[=] \frac{1}{\eta} \norm{\exponinv{x}(x^+)-\Gamma{y}{x}\exponinv{y}(y^+)}_x\\
                                                             &\circled{2}[\le] \frac{1}{\eta}\norm{\exponinv{x}(x^+)-\exponinv{x}(y^+)}_x + \frac{1}{\eta}\norm{\exponinv{x}(y^+)-\Gamma{y}{x}\exponinv{y}(y^+)}_x\\
     &\circled{3}[\le] \frac{1}{\eta}\dist(x^+,y^+)+ \frac{\zeta}{\eta} \dist(x,y)\circled{4}[\le] \frac{1+\zeta}{\eta}\dist(x,y).
 \end{align*}
    where $\circled{1}$ holds by \cref{corol:grad_of_moreau_envelope} and $\circled{2}$ is the triangular inequality. The bound of the first summand in $\circled{3}$ is \cref{lemma:underestimation_of_Hadamard_distance}, which holds in Hadamard manifolds, and the bound of the second summand holds by \cref{fact:hessian_of_riemannian_squared_distance}. Finally $\circled{4}$ uses the non-expansivity of the prox, cf. \cref{lemma:non_expansivity_of_the_prox}.
\end{proof}

We also have the following proof of the smoothness of the Moreau envelope for twice differentiable functions by making use of partial differential equations.

\begin{proposition}\label{lemma:Moreau_envelope_is_smooth}
   Let $f:\M \to \mathbb{R}$ for a manifold $\Mmin \in \Riemmin$ be a twice-differentiable g-convex function in some level set $\X \defi \{x\  |\ f(x) \leq f(p)\}$. Let $\zeta_{D} \defi \zeta_{\operatorname{diam}(\X)}$. The Riemannian Moreau envelope $M(y) \defi \min_{x\in\M}\{ f(x) + \frac{1}{2\eta}\dist(x, y)^2\}$ is $(\zeta_D/\eta)$-smooth in $\X$.
\end{proposition}

\begin{proof}
    Let $y \in \X$ and define $y^\ast$ as the $\argmin$ in the problem that defines $M$, that is $y^\ast = \prox(y)$. It is $y^\ast \in \X$ because any $z \not \in \X$ yields $f(z) + \frac{1}{2\eta}\dist(y, z)^2 > f(y) + \frac{1}{2\eta}\dist(y, y)^2 \geq f(y^\ast) + \frac{1}{2\eta}\dist(y, y^\ast)^2 $. Consider the first-order optimality condition of the problem in the definition of $M(y)$ and note $y^\ast$ is a function of $y$. We have
    \[
        \eta \nabla f(y^\ast) - \exponinv{y^\ast}(y) = 0.
    \]
    Differentiating with respect to $y$ we obtain
    \[
        \eta \nabla^2 f(y^\ast) \cdot \frac{\d y^\ast}{\d y} - \frac{\partial\exponinv{y^\ast}(y)}{\partial y^\ast} \cdot \frac{\d y^\ast}{\d y} - \frac{\partial\exponinv{y^\ast}(y)}{\partial y} = 0,
    \]
    from which we deduce
    \begin{equation}\label{eq:aux_differentiating_first_ord_opt_cond}
        \frac{\d y^\ast}{\d y} = \left(\eta \nabla^2 f(y^\ast) - \frac{\partial \exponinv{y^\ast}(y)}{\partial y^\ast}\right)^{-1}  \frac{\partial\exponinv{y^\ast}(y)}{\partial y}.
    \end{equation}

    Note that we can invert the matrix above since it is the Hessian of the strongly g-convex function $\eta f+\Phi_{y}$ at $y^\ast$. By \cref{corol:grad_of_moreau_envelope}, we obtain $\nabla M(y) = -\eta^{-1} \exponinv{y}(y^\ast)$.
    We differentiate again with respect to $y$ and use \eqref{eq:aux_differentiating_first_ord_opt_cond} to deduce:
    \[
     \nabla^2 M(y)  = -\frac{1}{\eta} \frac{\partial \exponinv{y}(y^\ast)}{\partial y^\ast} \left(\eta \nabla^2 f(y^\ast) - \frac{\partial \exponinv{y^\ast}(y)}{\partial y^\ast}\right)^{-1}\frac{\partial\exponinv{y^\ast}(y)}{\partial y} +\frac{1}{\eta} \frac{-\partial \exponinv{y}(y^\ast)}{\partial y}.
    \]
    Now, the second of the two matrices being added above is $\eta^{-1}\nabla^2 \Phi_{y^\ast}(y) \preccurlyeq (\zeta_{\dist(y, y^\ast)}/\eta) I$. The last inequality is due to \cref{fact:hessian_of_riemannian_squared_distance}. Note that $\nabla^2 M(y)$ is symmetric. The first matrix is symmetric and positive semidefinite. Indeed, we note that by \citet[Theorem 3.12]{lezcano2020curvature} one can conclude
    \[
        \frac{\sqrt{\abs{\kmin}}\dist(y, y^\ast)}{\sinh(\sqrt{\abs{\kmin}}\dist(y, y^\ast))} I\preccurlyeq \frac{\partial \exponinv{y}(y^\ast)}{\partial y^\ast},
    \]
    and the symmetric statement with respect to $y$ and $y^\ast$. For symmetric positive semidefinite matrices $A, B$ it is $A-B \preccurlyeq A$ and so $\nabla^2 M(y) \preccurlyeq \eta^{-1}\nabla^2 \Phi_{y^\ast}(y) \preccurlyeq (\zeta_{\dist(y, y^\ast)}/\eta) I$, so $M$ is $(\zeta_D/\eta)$-smooth.
\end{proof}

\section{Auxiliary results}
\label{sec:auxiliary-results}

\begin{proposition}\label{prop:bound-c}
  For $c>1$, and $T\in \mathbb{N}_{0}$ we have that
  \begin{equation*}
    \prod_{t=0}^T \frac{1}{1-(t+c)^{-2}} = \frac{c(c+T)}{(c-1)(c+T+1)}  \leq \frac{ c}{c-1}.
  \end{equation*}
\end{proposition}

\begin{proof}\linkofproof{prop:bound-c}
  We show $\prod_{t=0}^T \frac{1}{1-(t+c)^{-2}} = \frac{c(c+T)}{(c-1)(c+T+1)}$ by induction.
  The statement holds for $T=0$.
  Now assume that the statement holds for $T-1$.
  Then the statement also holds for $T$, which can be shown by noting that $\circled{1}$ below holds by the induction hypothesis and rearranging
   \begin{equation*}
      \prod_{t=0}^T \frac{1}{1-(t+c)^{-2}} \circled{1}[=]  \frac{c(c+T-1)}{(c-1)(c+T)}\frac{1}{1-(T+c)^{-2}}= \frac{c(c+T)}{(c-1)(c+T+1)}  \leq \frac{ c}{c-1}.
    \end{equation*}
\end{proof}

\section{Geometric Auxiliary Results}

In this section, we provide already established useful geometric results that we use in our proofs. Note that as we mentioned in the preliminaries, we may need to restrict the size of our set if positive curvature is present. Recall for instance that a g-convex function $f: \M\to\R$ defined over a compact manifold $\M$ must be constant \citep[Corollary 11.10]{boumal2023introduction}.

\begin{lemma}[Riemannian Cosine-Law Inequalities]\label{lemma:cosine_law_riemannian}
    For the vertices $x, y, p \in \M$ of a uniquely geodesic triangle of diameter $D$, we have
    \[
        \innp{\exponinv{x}(y), \exponinv{x}(p)} \geq \frac{\delta_{D}}{2} \dist(x,y)^2 + \frac{1}{2}\dist(p, x)^2 - \frac{1}{2}\dist(p, y)^2.
    \]
    and
    \[
        \innp{\exponinv{x}(y), \exponinv{x}(p)} \leq \frac{\zeta_D}{2} \dist(x,y)^2 + \frac{1}{2}\dist(p, x)^2 - \frac{1}{2}\dist(p, y)^2
    \]
\end{lemma}
See \cite{martinez2023acceleratedmin} for a proof.

\begin{remark}\label{remark:tighter_cosine_inequality}
    In spaces with lower bounded sectional curvature, if we substitute the constants $\zeta_D$ in the previous \cref{lemma:cosine_law_riemannian} by the tighter constant and $\zeta_{\dist(p,x)}$, the result also holds. See \cite{zhang2016first}.
\end{remark}

We note that if $\kmin < 0$, it is $\zeta_D = \Theta(1+D\sqrt{\abs{\kmin}})$ and therefore if $c$ is a constant, we have  $\zeta_{cD} = \bigo{\zeta_D}$. If $\kmin \geq 0$ it is $\zeta_r = 1$, for all $r\geq 0$, so it also holds $\zeta_{cD} = \bigo{\zeta_D}$.

\begin{lemma}\label{fact:hessian_of_riemannian_squared_distance}
    Consider a manifold $\M\in \Riem$ that contains a uniquely g-convex set $\X\subset \M$ of diameter $D<\infty$. Then, given $x,y\in\X$ we have the following for the function $\Phi_x:\M \to\mathbb{R}$, $y \mapsto \frac{1}{2}\dist(x, y)^2$:
\[
    \nabla \Phi_x(y) = -\exponinv{y}(x)\text{\quad \quad and \quad \quad} \delta_{D }\norm{v}^2 \leq \Hess  \Phi_x(y)[v, v] \leq \zeta_{D} \norm{v}^2.
\]
    Consequently, $\Phi_x$ is $\delta_{D}$-strongly g-convex and $\zeta_{D}$-smooth in $\X$. These bounds are tight for spaces of constant sectional curvature. 
\end{lemma}

See \cite{kim2022accelerated} for a proof, for instance. Note that the expression of $\nabla \Phi_x(y)$ along with \cref{lemma:cosine_law_riemannian} yields the smoothness and strong convexity inequalities.

\begin{lemma}\label{lemma:underestimation_of_Hadamard_distance}
    Let $\H$ be a Hadamard manifold of sectional curvature bounded below by $\kmin$. For any $x, y, z \in \M$, we have
    \[
        \norm{\exponinv{z}(x)- \exponinv{z}(y)}_z \leq \dist(x, y).
    \]
\end{lemma}

\begin{proof}
    Note that Hadamard manifolds are uniquely geodesic. Let $D$ be the diameter of the geodesic triangle with vertices $x$, $y$, and $z$. Using the Euclidean cosine theorem in $\Tansp{x} \M$ and \cref{lemma:cosine_law_riemannian} with $\delta_D = 1$, respectively, we have
\begin{align*}
\begin{aligned}
    2\innp{\exponinv{z}(x), \exponinv{z}(y)} &= \norm{\exponinv{z}(x)}^2 + \norm{\exponinv{z}(y)}^2 - \norm{\exponinv{z}(x)-\exponinv{z}(y)}^2, \\
    2\innp{\exponinv{z}(x), \exponinv{z}(y)} &\geq \dist(z, x)^2 + \dist(z, y)^2 - \dist(x, y)^2.
\end{aligned}
\end{align*}
Subtracting the first equation from the inequality below it, we obtain
\[
    0 \geq  \norm{\exponinv{z}(x)-\exponinv{z}(y)}^2 - \dist(x, y)^2.
\]
\end{proof}

\begin{lemma}\label{lemma:first_order_optimality}
    Let $\X\subseteq \M$ be a closed uniquely geodesically convex set and let $f:\M \to \mathbb{R}$ be a differentiable g-convex function in $\X$. Let $\xast \in \argmin_{x\in\X} f(x)$. We have
    \[
        \innp{\nabla f(\xast), \exponinv{\xast}(x)} \geq 0, \text{ for all } x \in \X.
    \]
\end{lemma}

\begin{proof}
  Let $f$ be g-convex and $\xast\in\argmin_{x\in \X}f(x)$.
    Let $F(t)\defi f(\gamma (t))$, where $\gamma $ is a geodesic such that $\gamma (0)=\xast$ and $\gamma(\dist(x, \xast))=x$.
  Then $F$ reaches its minimum at $t=0$ and we have that $0\le F'(0)=\langle \nabla f(\xast),\exponinv{\xast}(x)\rangle$.
\end{proof}

\begin{corollary}[Projection onto Geodesically Convex Sets]\label{corol:projection_condition}
    Consider a closed geodesically convex set $\mathcal{X}\subset \mathcal{M}$ in a manifold $\M\in\Riem$, and let $\tilde{x}\in\M$. If $\kmax > 0$, assume $\max_{x\in\X}\{\dist(x, \tilde{x})\} < \min\{\frac{\pi}{2\sqrt{\kmax}}, \inj(\tilde{x})\}$ where $\inj(x)$ is the injectivity radius. We have $P_{\mathcal{X}}(\tilde{x})$ is unique and equal to $\xast = \argmin_{x\in\X} \frac{1}{2}\dist(x, \tilde{x})^2$. Further, we have
\begin{equation*}
    \innp{ \exponinv{\xast}(\tilde{x}),\exponinv{\xast}(z)}_{\xast}\leq 0, \quad \forall z\in \mathcal{X}.
\end{equation*}
\end{corollary}

\begin{proof}
    Apply \cref{lemma:first_order_optimality} to the function $\Phi_{\tilde{x}}: \X \to\mathbb{R}, y \mapsto\frac{1}{2}\dist(\tilde{x}, y)^2$ whose gradient at the optimizer $\xast\in\X$ is $-\exponinv{\xast}(\tilde{x})$. Finally, since the assumption implies $\Phi_{\tilde{x}}$ is strictly convex in $\X$, we have $\dist(\xast, z) < \dist(\tilde{x}, z)$ for all $z \in \X\setminus\{\xast\}$, so indeed $\proj(\tilde{x})$ is unique and is $\xast$.
\end{proof}

\begin{lemma}[Geodesic averaging]\label{lem:g-avg}
  Let $\M\in\Riem$ and  $f:\M\rightarrow \mathbb{R}$ be a g-convex function in a g-convex set $\X\subset \M$ and let $\{x_1,\ldots,x_T\}$ be points in $\X$.
  The geodesic average $\bar{x}_{T}$ defined recursively by
  \begin{equation}
    \label{eq:g-avg}
    \bar{x}_1\gets x_{1}, \quad t\in \{1,\ldots,T-1\}:\;\bar{x}_{t+1}\gets \expon{\bar{x}_t}\left( \frac{w_{t+1}}{\sum_{j=1}^{t+1}w_j}\exponinv{\bar{x}_t}(x_{t+1}) \right)
 \end{equation}
 with $w_t> 0$ for all $t$ satisfies $f(\bar{x}_{T})\le \frac{1}{\sum_{t=1}^{T}w_{t}} \sum_{t=1}^{T}w_tf(x_t)$.
\end{lemma}

\begin{proof}
  We prove the statement by induction.
  The statement holds for $T=1$ by definition.
  Now assume that the statement holds for $T-1$, i.e., $f(\bar{x}_{T-1})\le \frac{1}{\sum_{t=1}^{T-1}w_{t}}\sum_{t=1}^{T-1}w_tf(x_t)$.
  We show that the statement holds for $T$ as well.
  By definition, $\bar{x}_{T}$ lies on the geodesic joining $\bar{x}_{T-1}$ and $x_T$.
    In particular, if we parametrize a geodesic segment joining $\bar{x}_{T-1}$ and $\gamma(1)=x_T$ as $\gamma:[0,1]\rightarrow \M$ with $\gamma(0)=\bar{x}_{T-1}$ and $\gamma(1)=x_T$, then $\gamma\left( \frac{w_T}{\sum_{t=1}^Tw_t}\right)=\bar{x}_T$.
  Hence by g-convexity of $f$ we have that,
  \begin{align*}
    f(\bar{x}_{T}) &\le \left(1-\frac{w_T}{\sum_{t=1}^Tw_t}\right)  f(\bar{x}_{T-1})+ \frac{w_T}{\sum_{t=1}^Tw_t} f(x_T)\\
                   &\circled{1}[\le]  \frac{1}{\sum_{t=1}^Tw_{t}}\sum_{t=1}^{T-1}w_tf(x_t)+\frac{w_T}{\sum_{t=1}^Tw_t} f(x_T)= \frac{1}{\sum_{t=1}^Tw_{t}}\sum_{t=1}^{T}w_tf(x_t)
  \end{align*}
  where $\circled{1}$ holds by the induction hypothesis and the fact that $1-\frac{w_T}{\sum_{t=1}^Tw_t}= \frac{\sum_{t=1}^{T-1}w_t}{\sum_{t=1}^Tw_{t}}$.
\end{proof}
\begin{corollary}\label{cor:g-avg-1}
  Let $w_t=1$ for all $t$, then the update rule simplifies to
  \begin{equation}\label{eq:g-avg-unif}
    \bar{x}_1\gets x_{1}, \quad t\in \{1,\ldots,T-1\}:\;\bar{x}_{t+1}\gets \expon{\bar{x}_t} \left( \frac{1}{t+1}\exponinv{\bar{x}_t}(x_{t+1}) \right)
  \end{equation}
  and we have $f(\bar{x}_{T})\le \frac{1}{T}\sum_{t=1}^Tf(x_t)$. We call this procedure uniform geodesic averaging.
\end{corollary}

\begin{lemma}\label{lemma:prox_is_frejer}
  Consider either a manifold $\M\in\Riem$, a set g-convex and closed set $\X\subset\M$ and a function $f\in \proxc[\X]$ or manifolds $\M,\N\in \Riem$, a set g-convex and closed set $\X\subset\M\times\N$ and a function $f\in \proxsp[\X]$.
  Let $\bar{z}^{*}=\prox(\bar{z})$ and $z^{*}\argmin_{x}f(x)$, then we have that
  \begin{equation*}
   \dist(\bar{z},\bar{z}^{*})\le \dist(\bar{z},z^{*}).
  \end{equation*}
\end{lemma}
See \citep[Lemma 37]{martinez2023acceleratedminmax} for a proof for $f\in\proxsp[\X]$ and \cite{martinez2023acceleratedmin} for the proof for $f\in\proxc[\X]$. Note that it was only stated in the context of Hadamard manifolds but it holds in the general Riemannian case.

\begin{proposition}\label{prop:karcher}
 The optimizer $\xast$ of \cref{eq:karcher} lies in $\X$.
\end{proposition}

\begin{proof}
  For the sake of contradiction, assume that $\xast\notin \X$.
  Denote by $\bar{x}^{*}\defi \proj[\X](\xast)$ the projection of $\xast$ onto $\X$.
  By \cref{corol:projection_condition}, we have $\dist(z,\bar{x}^{*})\le \dist(z,\xast)$ for all $z\in \X$. By definition, $y_i\in \X$ for all $i$, hence
  \begin{equation*}
F(\bar{x}^{*}) \leq \frac{1}{2}\sum_{i=1}^n\dist(\bar{x}^{*},y_i)^2\le  \frac{1}{2}\sum_{i=1}^n\dist(\xast,y_i)^2 = F(\xast),
  \end{equation*}
  which contradicts the assumption. Hence $\xast\in \X$ which concludes the proof.
\end{proof}

\subsection{Riemannian Generalized Danskin's theorem}
We note that the Generalized Danskin's theorem \citep[Proposition 4.5.1]{bertsekas2003convex} works in Riemannian manifolds. The reason is essentially that Danskin's theorem does not require convexity of the functions involved and we can talk about the functions retracted to the tangent space of $(x, y^\ast(x))$, apply the Euclidean Danskin's theorem and then use that the first-order information of the Riemannian function and the retracted function at $(x, y^\ast(x))$ is the same since retracting with the exponential map is a local isometry. Alternatively, one can see that the proof works without a problem in the Riemannian case.

\begin{proposition}[Riemannian Generalized Danskin's Theorem]\label{prop:riemannian_generalized_danskins_thm}
    Let $\M, \N$ be uniquely geodesic Riemannian manifolds and let $Y \subset \N$ be g-convex and compact. Let $f: \M \times \Y \to \mathbb{R}$ be a continuous function. Then, the function $\phi(x) \defi \max_{y\in\Y} f(x, y)$ has directional derivative
    \[
        \phi'(x;v) = \max_{y\in\Y(x)} f'(x, y; v)
    \]
    where $f'(x, y; v)$ is the directional derivative of $f(\cdot, y)$ at $x$ with direction $v$, and $\Y(x)$ is the set of maximizing points in the definition of $\phi$, that is $\Y(x) \defi \argmax_{y\in\Y}{f(x, y)}$. If $\Y(x)$ is a singleton $y^\ast$ and $f(\cdot, y^\ast(x))$ is differentiable at $x$, then $\psi$ is differentiable at $x$ and $\nabla \psi(x) = \nabla_x f(x, y^\ast(x))$.
\end{proposition}

Using the result above, we can provide the proof of \cref{corol:grad_of_moreau_envelope} about the gradient of the Moreau envelope. 
\gradientmoreauenvelope*

\begin{proof}\linkofproof{corol:grad_of_moreau_envelope}
    In order to compute $\nabla M(\hat{x})$ it is enough to consider the function $F(x, y) \defi f(y) +\indicator{\XX}(y) + \frac{1}{2\eta}\dist(x, y)^2$ for $x\in \hat{\XX} \defi\ball(\hat{x}, \delta)$ for any $\delta>0$. In such a case, it is easy to see that we can restrict to $y$ being in a compact $\Y$ in order to define $M(x) \defi \min_{y\in\XX}\{f(y) + \frac{1}{2\eta}\dist(x, y)^2\}$ for all $x\in \hat{\XX}$, that is, $M(x) \defi \min_{y\in\Y}\{f(y) + \frac{1}{2\eta}\dist(x, y)^2\}$. Indeed, consider $\Y = \{y\in\XX : \innp{v, \exponinv{\hat{x}}(y)} + \frac{1}{2\eta}\dist(\hat{x}, y)^2 \leq f(\hat{x})\}$ for a $v\in \partial f(\hat{x})$, then $\exponinv{\hat{x}}(\Y) \subseteq \Tansp{\hat{x}}\M$ is the level set of a quadratic plus $\indicator{X}$, which is compact and  so $\Y$ is compact as well. Note that by definition, if $y\not\in\Y$ then for all $x\in\hat{\XX}$ we have $F(\hat{x}, y) \geq \innp{v, \exponinv{\hat{x}}(y)} + \frac{1}{2\eta}\dist(\hat{x}, y)^2 > f(\hat{x}) = F(\hat{x}, \hat{x})$ so $y\not\in\argmin_{y\in\X} F(x, y)$. Thus, we can apply \cref{prop:riemannian_generalized_danskins_thm} with $\phi(x) = -M(x) = \max_{y\in\Y} -F(x, y)$ for $F$ defined in the compact $\hat{\XX}\times \Y$. The optimizer of $\max_{y\in\Y} -F(\hat{x}, y)$ is unique by strong convexity of $y\mapsto \frac{1}{2\eta}\dist(\hat{x}, y)^2$ and this point is $\prox[g](\hat{x})$. Thus, $M(\cdot)$ is differentiable at $\hat{x}$ and $\nabla M(\hat{x}) = \nabla_x F(\hat{x},\prox[g](\hat{x})) = -\frac{1}{\eta}\exponinv{\hat{x}}(\prox[g](\hat{x}))$, as desired.
\end{proof}

\section{Experiment Details}
\label{sec:experiment-details}

\paragraph{Computing the step sizes.}
By the g-strong convexity of $F$, the optimizer is unique.
Further the optimizer $\xast$ lies in $\X$, as we show in \cref{prop:karcher}. Recall that $\X$ was defined as a g-convex set containing all of the Karcher mean centers $y_i$.
We generate these centers $y_i$ as follows:
First, we define an anchor point $\bar{x}\in \M$, then we sample a $d$-dimensional vector $v_i\in \ball(0, r) \subset \Tansp{\bar{x}}$ uniformly, for some fixed radius $r$.
Then, we divide all but one of the randomly generated $v_i$ by a factor of $10$ and compute $y_i\gets \expon{\bar{x}}(v_i)$. This has an impact on the condition number and makes the problem harder.
This procedure ensures that $y_i\in \bar{\X}\defi \ball(\bar{x},r)$ for all $y_i$ and hence $\xast\in \bar{\X}$.
It follows that if we initialize our algorithm with $\xinit \in \bar{\X}$, we have $\R=\dist(x_0,x^{*})\le 2r$.
This allows us to upper bound $\R$ a priori by $\R\le 2\max_{i\in \{1,\ldots,n\}} \dist(y_i,\bar{x})$.

In order to compute an upper bound on the smoothness of $F$, we further need a lower bound of $\kappa_{\min}$.
Both manifolds have non-positive sectional curvature. In particular $\mathbb{H}^d$ has sectional curvature $-1$ everywhere.
Further, using the so-called affine-invariant metric,
\begin{equation*}
 \innp{X,Y}_P = \tr (P^{-1}XP^{-1}Y) \text{ for } P\in \mathcal{S}_+^d \text{ and } X,Y \in \Tansp{P}\mathcal{S}_+^{d},
\end{equation*}
the sectional curvature of $\mathcal{S}_+^d$ lies in $[-0.5,0]$ \citep[Prop I.1]{criscitiello2023accelerated}.

\paragraph{Step sizes for the Karcher Mean}
When using \RGD{} with $\eta = 1/(\L\zeta_{\bigo{\R}} )$ to solve \cref{eq:karcher}, we can ensure that the iterates stay in $\ball(\xast,\sqrt{6}r)$.
Hence $F$ is $\zeta_{\sqrt{6}r}$-smooth, which means that the step size is $\eta=\bigo{\frac{1}{\zeta_{r}^2}}$  and the convergence rate simplifies to $\bigotildel{\zeta_{\R}^2}$.
If we use \RGD{} $\eta = 1/\L$, the iterates stay in $\ball(\xast,(1+\sqrt{5}) r\zeta_{2r})$.
Hence, $F$ is $\zeta_D$ smooth in that set with $D = 2(1+\sqrt{5}) r\zeta_{2r}$ and since $\zeta_D=\bigo{\zeta_{\R}^2}$, we have $\eta =\bigo{ \frac{1}{\zeta_{r}^2}}$ and the convergence rate simplifies to $\bigotildel{\zeta_{\R}^2}$. For \RIPPA{}, in the first step of the subroutine we minimize the quadratic upper bound given by smoothness plus the regularizer in the tangent space $T_{x_t}\M$. For the rest of the steps, these are in different tangent spaces so we use a regular gradient step whose step size is given by the smoothness that is estimated as above. We performed 3 iterations in each subroutine.

\newpage
\section{Numerical Results}
\label{sec:numerical-results}
We present numerical results for the Karcher mean on $\mathbb{H}^d$ and $\mathcal{S}_{+}^d$ for different values of $n$ and $d$.
\begin{figure}[ht]
  \includegraphics[width=0.49\columnwidth]{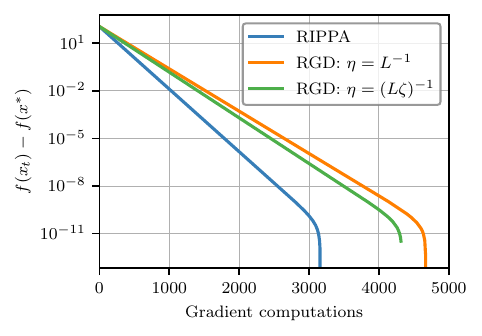}
  \includegraphics[width=0.49\columnwidth]{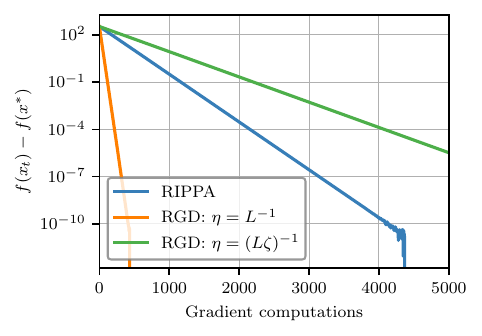}
  \caption{Corresponding plots in primal gap for \cref{fig:loss_H,fig:loss_SPD}.}
\end{figure}

\begin{figure}[ht]
  \includegraphics[width=0.49\columnwidth]{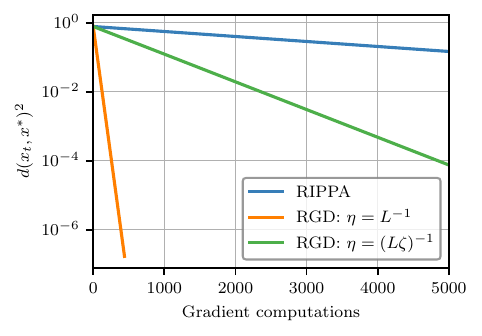}
  \includegraphics[width=0.49\columnwidth]{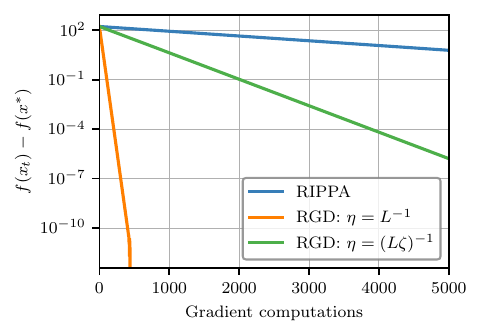}
  \caption{Karcher Mean on $\mathbb{H}^d$: $d=500$, $n=1000$, error in squared distance to the optimizer (left) and primal gap (right). }
\end{figure}

\begin{figure}[ht]
  \includegraphics[width=0.49\columnwidth]{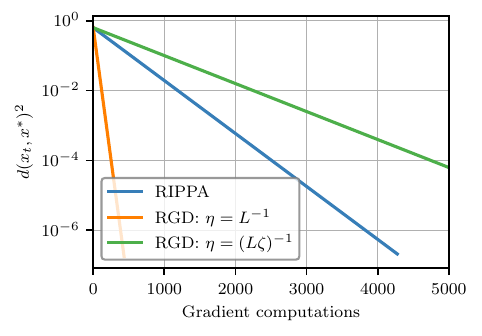}
  \includegraphics[width=0.49\columnwidth]{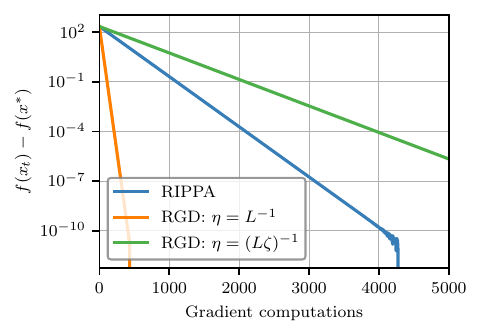}
  \caption{Karcher Mean on $\mathbb{H}^d$: $d=1000$, $n=500$, error in squared distance to the optimizer (left) and primal gap (right).}
\end{figure}

\begin{figure}[ht]
  \includegraphics[width=0.49\columnwidth]{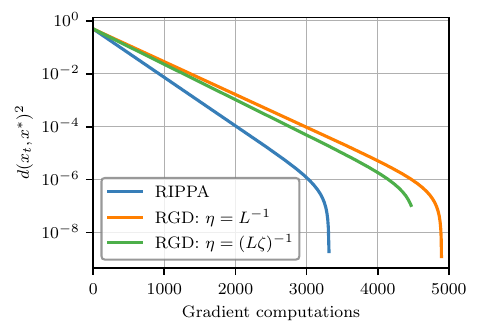}
  \includegraphics[width=0.49\columnwidth]{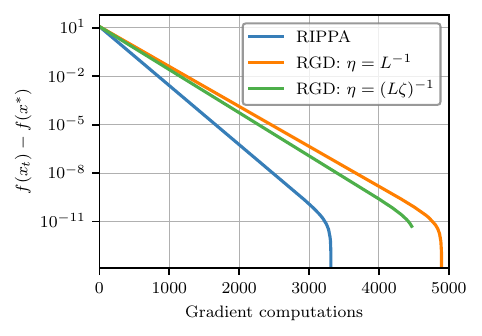}
  \caption{Karcher Mean on $\mathcal{S}_+^d$: $d=100$, $n=100$, error in squared distance to the optimizer (left) and primal gap (right).}
\end{figure}

\begin{figure}[ht]
  \includegraphics[width=0.49\columnwidth]{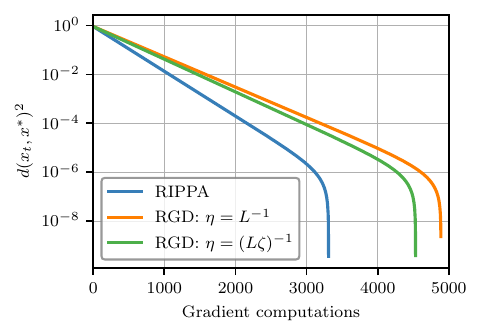}
  \includegraphics[width=0.49\columnwidth]{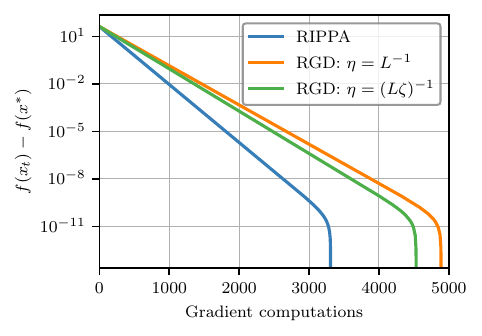}
  \caption{Karcher Mean on $\mathcal{S}_+^d$: $d=50$, $n=100$, error in squared distance to the optimizer (left) and primal gap (right).}
\end{figure}

\end{document}